\documentclass[11pt]{amsart}
\usepackage{fullpage}
\usepackage{amsmath}
\usepackage{amssymb}
\usepackage{amsthm}
\usepackage{color}
\usepackage{array}
\usepackage{xy}
\usepackage{setspace}

\newcommand{\CC}{\mathbb{C}}

\newcommand{\ZZ}{\mathbb{Z}}

\newtheorem{thm}{Theorem}[section]
\newtheorem{defn}[thm]{Definition}
\newtheorem{lem}[thm]{Lemma}
\newtheorem{prop}[thm]{Proposition}

\newtheorem{cor}[thm]{Corollary}
\newtheorem{conj}[thm]{Conjecture}
\theoremstyle{remark}
\newtheorem{rem}[thm]{Remark}
\newtheorem{ex}[thm]{Example}

\begin{document} 

\date{}
\title{Lower central series of a free associative 
algebra over the integers and finite fields}

\author{Surya Bhupatiraju, Pavel Etingof, David Jordan, William Kuszmaul, 
Jason Li}

\maketitle

\begin{abstract}
Consider the free algebra $A_n$ generated over $\Bbb Q$ by $n$
generators $x_1, \ldots, x_n$. Interesting objects
attached to $A = A_n$ are members of its lower central series, $L_i =
L_i(A)$, defined inductively by $L_1=A$, $L_{i+1}= [A,L_{i}]$, and their associated graded components $B_i = B_i(A)$ defined as
$B_i=L_i/L_{i+1}$. These quotients $B_i$ for $i \ge 2$, as well as the
reduced quotient $\bar{B}_1 = A/(L_2+A L_3)$, exhibit a rich geometric
structure, as shown by Feigin and Shoikhet \cite{FS} and later authors,
(\cite{DKM},\cite{DE},\cite{AJ},\cite{BJ}). 

We study the same problem over the integers $\mathbb{Z}$ and finite
fields $\mathbb{F}_p$. New phenomena arise, namely, torsion in $B_i$
over $\mathbb{Z}$, and jumps in dimension over $\mathbb{F}_p$. We
describe the torsion in the reduced quotient $\bar{B}_1$ and
$B_2$ geometrically in terms of the De Rham cohomology of
$\ZZ^n$. As a corollary we obtain a complete description of
$\bar{B}_1(A_n(\mathbb{Z}))$ and $\bar{B}_1(A_n(\mathbb{F}_p))$, 
as well as of $B_2(A_n(\mathbb{Z}[1/2]))$ and $B_2(A_n(\mathbb{F}_p))$,
$p>2$. We also give theoretical and experimental results for $B_i$
with $i>2$, formulating a number of conjectures and questions on their basis.
Finally, we discuss the supercase, when some of the generators 
are odd and some are even, and provide some theoretical 
results and experimental data in this case. 
\end{abstract} 

\section{Introduction}

The lower central series $\lbrace{L_i(A),i\ge 1\rbrace}$, of a noncommutative algebra $A$ is defined inductively by $L_1(A)=A$, and $L_{i+1}(A)=[A,L_i(A)]$. In other words, $L_i(A)$ is spanned by "$i$-commutators"
$[a_1,[a_2,...[a_{i-1}, a_i]...]]$, where $a_1,...,a_i\in A$. It is
interesting to consider the successive quotients
$B_i(A)=L_i(A)/L_{i+1}(A)$, i.e., $i$-commutators modulo
$i+1$-commutators. The role of the $B_i(A)$ is, roughly, that they
characterize, step by step, the deviation of $A$ from being
commutative; this is somewhat similar to the quasiclassical
perturbation expansion in quantum mechanics.

The quotients $B_i(A)$ can be put together in a direct sum
$B(A)=\oplus_{i\ge 1}B_i(A)$. This is a graded Lie algebra: we have $[B_i(A),B_j(A)]\subset B_{i+j}(A)$. It turns out that this
Lie algebra has a big central part (i.e., a part commuting with
everything) in degree 1, namely the image $I$ of $AL_3(A)$
in~$B_1$. Denote $B_1(A)/I$ by $\bar B_1(A)$. Then we have a graded
Lie algebra $\bar B(A)=\bar B_1(A)\bigoplus\oplus_{i\ge 2}B_i(A)$,
generated by degree 1 (i.e., $\bar B_1(A)$).

The lower central series of the free algebra $A_n=A_n(\mathbb{Q})$ over the
field of rational numbers $\mathbb{Q}$ was studied in a number of
papers \cite{DE}, \cite{DKM}, \cite{AJ}, \cite{BJ}, following the
pioneering paper \cite{FS}. In this paper, Feigin and Shoikhet
observed a surprising fact: the graded spaces $\bar{B}_1(A_n)$ and
$B_i(A_n)$, $i\ge 2$, have polynomial, rather than exponential growth,
even though the dimensions of the homogeneous
components of $A_n$ itself grow exponentially. This stems from the
fact that, even though $A_n$ is "the most noncommutative" of all
algebras, the spaces $\bar{B}_1(A_n)$ and $B_i(A_n)$, $i\ge 2$, can be
described in terms of usual, ``commutative" geometry, namely in terms
of tensor fields on the $n$-dimensional space. This allows one to completely describe $\bar
B_1(A_n)$ and $B_2(A_n)$ in terms of differential forms \cite{FS}
and $B_3(A_n)$ as well as $B_i(A_n)$ for some $n$ and $i>3$ in terms
of more complicated tensor fields \cite{DE}, \cite{DKM}, \cite{AJ},
\cite{BJ}. However, the problem of describing of $B_i(A_n(\Bbb Q))$
for all $i,n$, and in particular of computation of the Hilbert series
of these spaces still remains open.

This paper is dedicated to the study of $\bar B_1(A_n)$ and $B_i(A_n)$ over the integers $\Bbb Z$ and over a finite field $\Bbb F_p$. In this case, rich new structures emerge. Namely, in the case of $\Bbb Z$, the groups $\bar B_1(A_n)$, $B_i(A_n)$ develop torsion, and it is interesting to study the pattern of this torsion. In the case of $\Bbb F_p$ the dimensions of certain homogeneous subspaces of $\bar B_1(A_n)$, $B_i(A_n)$ differ from those over $\Bbb Q$, and the patterns of these discrepancies are of interest. 

Our main results are as follows. First of all, we give a complete description of $\bar{B}_1(A_n(\Bbb Z))$ in terms of differential forms, and of the torsion in $\bar B_1(A_n(\mathbb{Z}))$ in terms of the De Rham cohomology of the n-dimensional space over $\mathbb{Z}$. Since $\bar B_1(A_n(\Bbb F_p))=\bar B_1(A_n(\Bbb Z))\otimes \Bbb F_p$, this completely describes  $\bar B_1(A_n(\Bbb F_p))$ in terms of differential forms. Since the Lie algebra 
$\bar B$ is generated by $\bar B_1$, 
this description implies a uniform bound for dimensions of 
the homogeneous subspaces, which depends only on $i$ and $n$ (``polynomial growth'' 
for $B_i(A_n(\Bbb F_p))$). 
Also, we give a description 
of the torsion in $B_2(A_n(\Bbb Z))$, in terms of the De Rham cohomology over the integers
(conjectural for $2^r$-torsion). Since 
$B_2(A_n(\Bbb F_p))=B_2(A_n(\Bbb Z))\otimes \Bbb F_p$, this describes 
$B_2(A_n(\Bbb F_p))$ in terms of differential forms. Further, we give an explicit formula for some torsion elements in $B_2(A_3(\Bbb Z))$, and show that they are a basis of torsion.

We also present some theoretical and computer results (using the MAGMA package, \cite{M}) for torsion in $B_i(A_n(\Bbb Z))$ for $i>2$, and 
propose a number of conjectures based on them. 
In particular, we conjecture that the degree of any torsion element with respect to each variable 
is divisible by the order of the torsion. 
Finally, we provide some computer data and theoretical results for 
torsion in $B_i$ in the supercase, i.e. for 
free algebras $A_{n,k}$ in $n+k$ generators, where 
the first $n$ generators are even and the last $k$ generators are odd, and 
formulate a number of open questions based on this data.  However, the pattern of torsion in higher $B_i$ remains mysterious even at the conjectural level, and will be the subject of further investigation.  

The organization of the paper is as follows. 
In Section 2, we discuss preliminaries. 
In Section 3, we describe the structure of 
$\bar B_1(A_n)$. In Section 4, we describe the structure of $B_2(A_n)$. 
In Section 5, we discuss the experimental data and conjectures 
for the structure of $B_i(A_n(\Bbb Z))$ for $i>2$. 
In Section 6, we discuss the theoretical results, experimental data and open problems in the supercase. 
Finally, in Section 7, we outline some directions of further research. 

{\bf Acknowledgments.} The work of P.E. was partially supported 
by the NSF grant DMS-1000113. The work of D.J. was supported by the NSF grant DMS-1103778.  S.B., W.K., and J.L.  
are grateful to the PRIMES program at the MIT Mathematics Department, where this research was done. 

\section{Preliminaries}
\subsection{The lower central series}

Let $A$ be an algebra over a commutative ring $R$.
In this paper, we will consider algebras over $R = \mathbb{Q},\mathbb{F}_p,\mathbb{Z}$.

\begin{defn}
The lower central series filtration of $A$ is defined as $L_1(A):=A, L_{k+1}(A):= [A,L_{k}(A)]$ for $k \in \mathbb{N}$.
\end{defn}
In other words, we have \[ L_k = \text{Span}_R \{[a_1, [a_2, \ldots, [a_{k-1}, a_k]] \ldots] | a_1,\ldots, a_k\in A\}. \]

\begin{defn}
The associated graded Lie algebra of $A$ is \[ B(A):= \displaystyle\bigoplus_{k \ge 1}{} B_k(A) \] where  $B_k(A):=L_k/L_{k+1}$ for all $k \in \mathbb{N}$.
\end{defn}
\begin{defn}
The two-sided ideals $M_k(A)$ are defined as $M_k(A) := AL_k$ for $k \in \mathbb{N}$.
\end{defn}

\begin{lem} \label{lem:FSLemma1}
$[A,M_3] \subset L_3$.
\end{lem}
\begin{proof}
Lemma 2.2.1 from \cite{FS} asserts this for algebras over $\CC$. However, the formulas in the proof only involve $\pm 1$ coefficients, so the proof applies for an arbitrary ring $R$. 
\end{proof}

\begin{defn}
The abelian group $\bar{B}_1$ is the quotient $\bar{B}_1:=A/(L_2+M_3)$. 
\end{defn}
Thus, $\bar{B}_1$ is the quotient of $B_1$ by the image $I$ of $M_3$ in $B_1$, which by Lemma ~\ref{lem:FSLemma1} is central in $B(A)$. 

\begin{cor}
The graded module 
$$
\bar{B}(A) := \bar{B}_1 \oplus \displaystyle \bigoplus_{k \ge 2} B_k
$$
is a graded Lie algebra generated in degree 1.  
\end{cor}

\begin{defn}
The free algebra $A_n(R)$ of rank $n$ over $R$ is the free $R$-module generated by all finite words in the letters $x_1,\ldots, x_n$; the multiplication is given by concatenation of words.
\end{defn}

\subsection{De Rham cohomology over the integers}

We will see that torsion in $\bar B_1(A_n(\Bbb Z))$ and $B_2(A_n(\Bbb Z))$ is related to the De Rham cohomology 
of the $n$-dimensional space over the integers. So let us discuss how to compute this cohomology. 

Denote by $\Omega(R^n)$ the module of differential forms in $n$ variables with coefficients in $R$. 
We have a rank decomposition $\Omega(R^n)=\oplus_{i=0}^n \Omega^i(R^n)$, 
and the De Rham differential $d: \Omega^i(R^n)\to \Omega^{i+1}(R^n)$, which defines the 
De Rham complex of the $n$-dimensional space over $R$.   We will use the notation $\Omega_{cl}$ for closed forms, $\Omega_{ex}$ for exact forms, 
$\Omega^{ev}$ for even forms, $\Omega^{odd}$ for odd forms, $\Omega^+$ for positive rank 
forms, $\Omega^{\ge k}$ for forms of rank $\ge k$, etc. 

If $R$ is a $\Bbb Q$-algebra 
then by the Poincar\'e lemma, the De Rham complex is acyclic in degrees $i>0$, 
and its zeroth cohomology $H^0$ consists of constants (i.e., equals $R$).
However, this fails for general rings. In this section we compute the cohomology 
of this complex for $R=\Bbb Z$. 

\begin{lem}\label{omega1} We have:
$$H^k(\Omega(\ZZ))[m] \cong \left\{\begin{array}{ll} \ZZ,& k=0, \, m=0\\ \ZZ/m, & k=1, \, m > 0, \\0,& \mathrm{otherwise} \end{array}\right.$$
(Here $[m]$ denotes the degree $m$ part). 
\end{lem}

\begin{proof} Straightforward computation.\end{proof}

\begin{cor}\label{kunnethcor} Let $H^r_n := H^r (\Omega (\ZZ^n))$. We have a (noncanonically) split short exact sequence in cohomology,
\[ 0 \to H^{r}_{n-1} \oplus H^{r-1}_{n-1} 
\otimes H^1_1 \to H^r_n \to \mathrm{Tor}(H^r_{n-1}, H^1_1) \to 0. \]

Moreoever, this is a short exact sequence of graded abelian groups, with respect to the natural $\ZZ^n$-grading on each term.  
\end{cor}
\begin{proof}
Apply the 
K\"unneth Theorem and Lemma \ref{omega1} to the decomposition 
$$
\Omega(\ZZ^n)=\Omega(\ZZ^{n-1})\otimes \Omega(\ZZ).
$$ 
\end{proof}

\begin{cor}\label{Hev-description} (see also \cite{FR}) Suppose that all $m_i>0$.  Then we have an isomorphism:
$$\text{H}^{i}(\Omega(\Bbb Z^n))[m_1, \ldots , m_n] \cong (\ZZ/gcd(m_1,\ldots,m_n))^{{n-1}\choose{i-1}},
$$ 
where $[m_1,...,m_n]$ denotes the part of multidegree $m_1,...,m_n$.  
\end{cor}

\begin{rem}
Note that this describes the cohomology completely, since 
the case when $m_i=0$ for some $i$ reduces to a smaller number of variables. 
\end{rem}

\begin{proof}
We choose a splitting of the short exact sequence of Corollary \ref{kunnethcor}, and apply Lemma \ref{omega1} to obtain:
$$H^r_n \cong H^{r}_{n-1} \oplus \bigoplus_{m\geq 1} \left(H^{r-1}_{n-1}\otimes \ZZ/m \oplus \mathrm{Tor}(H^r_{n-1}, \ZZ/m)\right).$$
Note that each $\ZZ/m$ in the sum above is in multi-degree $(0,\ldots,0,m)$.  Thus the $(m_1,\ldots,m_n)$-graded part of the above short exact sequence reads:
$$H_n^i[m_1,...,m_n]\cong H_{n-1}^i[m_1,...,m_{n-1}]/m_n\oplus H_{n-1}^{i-1}[m_1,...,m_{n-1}]/m_n.$$
The formula now follows by induction on $n$ and $i$, and the identity ${{r+1}\choose{i+1}}={{r}\choose{i}} + {{r}\choose{i+1}}$.
\end{proof}

\subsection{Universal coefficient formulas for $\bar{B}_1$ and $B_2$}

Let $p$ be a prime and $\mathbb{Z}^N \supset A \supset B$ be abelian groups. 
For an abelian group $A$, let ${\rm tor}_p(A)$ denote the $p$-torsion of $A$. 

\begin{lem}\label{lem:Universal}
Let $A_p$ and $B_p$ be the images of $A$ and $B$ in ${\Bbb F}_p^N$. 
If ${\rm tor}_p(\mathbb{Z}^N / A)=0$ (in particular, if $\mathbb{Z}^N / A$ is free) 
then the natural map \[ A_p / B_p \rightarrow (A/B) \otimes \mathbb{F}_p \] is an isomorphism. 
\end{lem}
\begin{proof}
We have $\mathbb{Z}^N/A = \dfrac{\mathbb{Z}^N / B}{A/B}$. Since ${\rm tor}_p(\mathbb{Z}^N/A)=0$, this implies that we have an isomorphism of $p$-local parts
$(\mathbb{Z}^N /B)\otimes {\Bbb Z}_{(p)} = (A/B \oplus \mathbb{Z}^N / A)\otimes \Bbb Z_{(p)}$, 
where $\Bbb Z_{(p)}$ is the ring of rational numbers whose denominator is coprime to $p$. 
When we tensor the equality with $\mathbb{F}_p$, we see that $(A/B) \otimes \mathbb{F}_p$ is the kernel of the projection map $(\mathbb{Z}^N / B) \otimes \mathbb{F}_p \rightarrow (\mathbb{Z}^N / A) \otimes \mathbb{F}_p$, which is $A_p / B_p$, as desired. 
\end{proof}

\begin{cor}\label{unco} 
We have 
\begin{enumerate}
\item[\begin{small}1)\end{small}] $\bar{B}_1(A_n(\mathbb{F}_p)) = \bar{B}_1 (A_n(\mathbb{Z})) \otimes \mathbb{F}_p$;
\item[\begin{small}2)\end{small}] $B_2(A_n(\mathbb{F}_p)) = B_2(A_n(\mathbb{Z})) \otimes \mathbb{F}_p$.
\end{enumerate}
\end{cor}
\begin{proof}
\
\begin{enumerate}
\item[\begin{small}1)\end{small}]Apply Lemma ~\ref{lem:Universal} for $\mathbb{Z}^N = A = A_n[m]$ and $B = (L_2 + M_3)[m]$, where $[m]$ means total degree $m$. 
\item[\begin{small}2)\end{small}] Apply Lemma ~\ref{lem:Universal} for $\mathbb{Z}^N = A_n[m]$, $A = L_2(A_n)[m]$, and $B = L_3(A_n)[m]$. $\mathbb{Z}^N / A = B_1[m]$ is freely spanned by cyclic words of length $m$, so it is free, and the lemma applies. 
\end{enumerate}
\end{proof}

\begin{rem}
This corollary is false for higher $B_i$. 
For instance, one can show that
$$
\dim (B_4(A_3(\Bbb Z))[2,2,2]\otimes \Bbb F_2)\ne \dim B_4(A_3(\Bbb F_2))[2,2,2].
$$
(see Subsection \ref{discr}). 
\end{rem}

\section{The Structure of $\bar{B}_1(A_n)$}

It will be useful to adapt the presentation of $A_n/M_3$ given in \cite{FS} to general base ring $R$.
This is accomplished by the following proposition. 


\begin{prop}\label{relaa}
The algebra $A_n/M_3$ is the algebra generated by 
$x_1, \ldots x_n$, and $u_{ij}$ for $i, j \in [1 \ldots n]$ and $i \neq j$, with the following relations:
\begin{enumerate}
\item[\begin{small}1)\end{small}] $u_{ij} = [x_i, x_j]$, and so $u_{ij} + u_{ji} = 0$;
\item[\begin{small}2)\end{small}] $[x_i, u_{jk}] = 0$ for all $i, j, k$;
\item[\begin{small}3)\end{small}] $u_{ij}$ commute with each other: $[u_{ij}, u_{kl}] = 0$;
\item[\begin{small}4)\end{small}] $u_{ij}u_{kl} = 0$ if any of the $i, j, k, l$ contain repetitions;
\item[\begin{small}5)\end{small}] $u_{ij}u_{kl}=-u_{ik}u_{jl}$ if $i, j, k, l$ are all distinct.
\end{enumerate}
\end{prop}

\begin{proof}
Relation (1) is the definition of $u_{ij}$, 
and relations (2),(3) are obvious, so we will prove 
only relations (4),(5). 

4) We want to show that $[x,y][x,z]$ is in $M_3$. 
But, 
$$[x,y][x,z]=[[x,y]x,z]-[[x,y],z]x=[[x,yx],z]-[[x,y],z]x,
$$ 
which is in $M_3$. 

5) We want to show that $[x,y][z,t]+[x,z][y,t]$ is in  $M_3$. 
Modulo $M_3$, this is 
$$
[x,y[z,t]+z[y,t]]=[x,[yz,t]]+[x,z[y,t]-[y,t]z],
$$
which is obviously in $L_3$ (not just $M_3$).

Also, it is easy to show using relations (1)-(5) that 
in the quotient of the algebra $A_n$ by these relations, 
we have $[[a,b],c]=0$ for all $a,b,c$. Indeed, 
using the Jacobi identity and the Leibniz rule, it suffices to show 
that $[x_i,[x_j,c]]=0$. This is straightforward, substituting 
$c=x_{i_1}...x_{i_m}$. The proposition is proved. 
\end{proof}

For $I\subset \{1,\ldots, n\}$, of even cardinality $k$, $I=\lbrace{i_1,...,i_k\rbrace}$, 
$i_1<...<i_k$, define $u_I:=u_{i_1,i_2}\cdots u_{i_{k-1},i_k}$.

\begin{prop}\label{anm3} $A_n/M_3$ is a free abelian group with basis 
$x_1^{k_1}\cdots x_n^{k_n} u_I$, for $k_i\geq 0$, and $I$ as above.
\end{prop}

\begin{proof}
It is clear that the given elements are a spanning set. 
Let us show that they are a basis. It is enough to show it 
over $R=\Bbb Z$, hence over $R=\Bbb Q$, which is known from 
\cite{FS}.  
\end{proof} 

Here is a coordinate-free description of $A/M_3$. 
Let $V$ be a finitely generated free $R$-module, 
and $A=T_RV$ be the tensor algebra of $V$. 
Define the algebra $\widetilde\Omega(V)$ 
to be the quotient of the tensor algebra $T_R(V\oplus V)$ 
with generators $x_v,y_v$, $v\in V$ ($R$-linearly depending on $v$)
by the relations 
$$
[x_v,x_u]=y_vy_u, [x_v,y_u]=0, v,u\in V.
$$
Note that this implies that $y_vy_u=-y_uy_v$ and $y_v^2=0$. 
We have a decomposition $\widetilde{\Omega}(V)=\widetilde{\Omega}^{odd}(V)\oplus \widetilde{\Omega}^{ev}(V)$ 
into parts of odd and even degree with respect to the $y$-grading. 
Also we have a differential on $\widetilde{\Omega}(V)$ given by $dx_v=y_v$. 

Also, let $\Omega(V)$ be the algebra of differential forms 
on $V^*$, generated by $x_v$ and $y_v$ with relations 
$$
[x_v,x_u]=0,\ [x_v,y_u]=0,\ y_vy_u=-y_uy_v,\ y_v^2=0 
$$
and differential $dx_v=y_v$.
If $R$ contains $1/2$, we can define the Fedosov 
$*$-product on $\Omega(V)$ by 
$a*b=ab+\frac{1}{2}da\wedge db$. 
Denote $\Omega(V)$ equipped with this product by $\Omega(V)_*$.

\begin{prop}\label{phimap}
(i) If $R$ contains $1/2$ then there is a unique differential graded algebra homomorphism 
$\phi: \widetilde{\Omega}(V)\to \Omega(V)_*$ such that $\phi(x_v)=x_v$,
and it is an isomorphism.

(ii) Over any $R$, one has ${\rm gr}(\widetilde{\Omega}(V))=\Omega(V)$,
where the associated graded is taken with respect to the 
descending filtration by degree in the $y$-variables. 

(iii) There exists a unique algebra homomorphism $\zeta: A/M_3\to \widetilde{\Omega}^{ev}(V)$, 
such that $\phi(v)=x_v$. Moreover, $\phi$ is an isomorphism. 
\end{prop}

\begin{proof}
(i) It is easy to check that the relations $[x_v,x_u]=y_vy_u$
and $[x_v,y_u]=0$ are satisfied in $\Omega(V)_*$, so, there is a unique 
homomorphism $\phi$, which is surjective. 
Also, it is clear that the Hilbert series 
of $\widetilde{\Omega}(V)$ is dominated by the Hilbert series of $\Omega(V)$, 
so $\phi$ is injective, hence an isomorphism. 

(ii) It is easy to see that there is a surjective homomorphism 
$\theta: \Omega(V)\to {\rm gr}\widetilde{\Omega}(V)$. 
So it suffices to show that it is an isomorphism over $\Bbb Q$, 
which follows from (i) by comparison of Hilbert series.  

(iii) It is shown similarly to the proof of Proposition \ref{relaa} 
that in the algebra $\widetilde{\Omega}(V)$, 
one has $[[[a,b],c]=0$, so that $\zeta$ exists and is surjective. 
But it follows from (ii) and Proposition \ref{relaa} 
that the Hilbert series of the two algebras are the same, so
$\zeta$ is injective, i.e. an isomorphism.   
\end{proof}
 
Thus, $\widetilde{\Omega}(V)$ is a quantization of 
the DG algebra of differential forms $\Omega(V)$. 
Nevertheless, over a ring $R$ not containing $1/2$, in particular 
$R=\Bbb F_2$, they are not isomorphic even 
as graded $GL(V)$-modules. Namely, the degree 
2 component of $\widetilde{\Omega}(V)$ is $V\otimes V$ 
(the corresponding isomorphism is defined by $v\otimes u\mapsto x_vx_u$, 
$v,u\in V$), while for $\Omega(V)$ it is $S^2V\oplus \wedge^2V$, which is not the same thing as $V\otimes V$ 
in characteristic $2$. However, we will need to work with rings not containing $1/2$ (such as $\Bbb Z$ and $\Bbb F_2$). 
For this reason we will use a poor man's version of $\phi$, the map $\varphi$, introduced in the following proposition. 
Note that it is not an algebra map and is not $GL(n)$-invariant. 

\begin {prop}\label{coro} There exists a unique isomorphism of $R$-modules $\varphi:A_n/M_3\to\Omega^{ev}$ 
such that \[ \varphi(x_1^{k_1}\cdots x_n^{k_n}u_I) = x_1^{k_1}\cdots x_n^{k_n} dx_I, \]
where $dx_I = d_{x_{i_1}} \wedge d_{x_{i_2}} \wedge \ldots \wedge d_{x_{i_{k-1}}} \wedge d_{x_{i_k}}$. 
Moreover, one has $\varphi([a,x_i])=d\varphi(a)\wedge dx_i$. 
\end{prop}

\begin{proof} 
The first statement is a direct consequence of Proposition \ref{anm3}, and the second one follows by an easy 
direct computation. 
\end{proof} 

\begin{thm}\label{thm:Thm1}
The map $\varphi$ induces an isomorphism $\varphi: L_2(A/M_3) \rightarrow \Omega^{ev}_{ex}$.
\end{thm}

\begin{proof} We will need the following lemma. 

\begin{lem}\label{degg1} We have $L_2(A_n)=\sum_{i=1}^n [A_n,x_i]$.
\end{lem}

\begin{proof}
This follows from the identity $[a,bc]+[b,ca]+[c,ab]=0$.
\end{proof}

By this lemma, $L_2(A_n)$  is spanned by $[a,x_i]$, $a\in A_n$.
But by Proposition \ref{coro},  
$\varphi([a,x_i])=d\varphi(a)\wedge dx_i$.
Thus, $\varphi(L_2)$ is contained in $\Omega^{ev}_{ex}$.
On the other hand, $\Omega^{ev}_{ex}$ is spanned by elements of the form
$\omega=df\wedge dx_{i_1}\wedge...\wedge dx_{i_{2r+1}}$. So if $a=\varphi^{-1}(fdx_{i_1}\wedge...\wedge dx_{i_{2r}})$,
then $\varphi([a,x_{i_{2r+1}}])=\omega$, implying the opposite inclusion.
\end{proof} 

\begin{cor}\label{deco} 
Over any base ring $R$, one has 
$$
\bar{B}_1(A_n) =\Omega^{ev}(R^n)/\Omega^{ev}_{ex}(R^n).
$$
In particular, for $R=\Bbb Z$, one has 
$$
\bar{B}_1(A_n)=H^{ev,+}(\Omega(\ZZ^n))\oplus\Omega^{ev}(\ZZ^n)/\Omega^{ev,+}_{cl}(\ZZ^n),$$
where $H^{ev,+}:=H^2\oplus H^4\oplus...$ is the even De Rham cohomology of $\Bbb Z^n$. 
\end{cor}

\begin{proof}
It is clear that 
for any algebra $A$, 
$\bar{B}_1(A) \cong B_1(A/M_3)$.
Therefore, by Theorem \ref{thm:Thm1},
$\bar{B}_1(A_n) =\Omega^{ev}/\Omega^{ev}_{ex}$.
So we have a short exact sequence
$$
0\to H^{ev,+}(\Omega(\ZZ^n))\to \bar{B}_1(A_n) \to \Omega^{ev}(\ZZ^n)/\Omega^{ev,+}_{cl}(\ZZ^n)\to 0,
$$
But it is easy to see that the quotient is a free group
(indeed, if $m\omega$ is a closed form for an integer $m$, then $\omega$ is closed as well). This implies 
the statement.
\end{proof}

Now, the Poincare lemma over $\mathbb{Q}$ implies that $H^{ev,+}(\Omega(\ZZ^n))$ is torsion.  
Thus, we have

\begin{thm} \label{barB1thm} The torsion in $\bar B_1(A_n\ZZ))$ is given by the equality
$$
{\mathrm{tor}}\bar{B}_1(A_n(\ZZ))= H^{ev,+}(\Omega(\ZZ^n)) = H^2(\Omega(\ZZ^n)) \oplus H^4(\Omega(\ZZ^n)) \oplus \ldots = \bigoplus^{\lfloor\frac{n}{2}\rfloor}_{k=1} H^{2k} (\Omega(\ZZ^n)).$$
\end{thm}

Combining with Corollary \ref{Hev-description}, we have a complete description of the torsion part of $\bar{B}_1(A_n(\ZZ))$.

\begin{prop}\label{b1an} If $m_1,...,m_n>0$ then the torsion in $\bar B_1(A_n(\Bbb Z))[m_1,...,m_n]$ is isomorphic to 
$(\Bbb Z/{\rm gcd}(m_1,...,m_n))^{2^{n-2}}$. 
\end{prop}

\begin{cor}\label{b1a2}
For $q,r>0$, the torsion in $\bar B_1(A_2(\mathbb{Z}))[q,r]$ is isomorphic to $\ZZ/{\rm gcd}(q,r)$ and 
spanned by the element $x^{q-1}y^{r-1}[x,y]$. 
\end{cor}

Proposition \ref{b1an} allows us to explicitly compute $\bar{B}_1(A_n(\mathbb{F}_p))$ for all $p$.

\begin{cor}\label{b1anp} 
If all $m_i$ are positive, then $\dim \bar B_1(A_n(\Bbb F_p))[m_1,...,m_n]$
is $2^{n-2}$ if there exists $i$ such that $m_i$ is not divisible by $p$, and $2^{n-1}$ 
otherwise (i.e., if all $m_i$ are divisible by $p$). 
\end{cor} 

\begin{proof}
By Corollary \ref{unco}, $\bar B_1(A_n(\Bbb F_p))=\bar B_1(A_n(\Bbb Z))\otimes \Bbb F_p$. 
It follows from the results of \cite{FS} that the rank of the free part 
of $\bar B_1(A_n(\Bbb Z))[\bold m]$ is $2^{n-2}$. So the result follows from 
Corollary \ref{b1an}.   
\end{proof}

\section{The Structure of $B_2(A_n)$}

\subsection{Torsion elements in $B_2(A_n(\Bbb Z))$}

In this section we study the torsion in $B_2(A_n(\Bbb Z))$. 
We will show below that for $n=2$ there is no torsion, 
so the first interesting case is $n=3$. 
In this subsection we describe the torsion in the case $n=3$. 
Later we will give a more general result 
which applies to any $n$; however, first we explicitly 
work out the cases $n=3$ and $n=2$ for the reader's convenience.

Let us denote the generators of $A_3(\Bbb Z)$ by $x,y,z$. 
Let $s, q, r$ be positive integers, and $T(s, q, r) = [z, z^{s-1}x^{q-1}y^{r-1}[x,y]] \in B_2(A_3(\ZZ))$.

\begin{thm} \label{TorsionTheorem}
\
\begin{enumerate}
\item[\begin{small}1) \end{small}] The element $T(s, q, r)$ is torsion of order dividing $\gcd(s, q, r)$. 
\item[\begin{small}2) \end{small}] If $\gcd(s, q, r) = 2$ or $3$, the order of $T(s, q, r)$ is exactly equal to $\gcd(s, q, r)$. 
\end{enumerate}
\end{thm}
\begin{proof}
\
\begin{enumerate}
\item[\begin{small}1) \end{small}] 
We start with the identity 
\begin{equation}\label{ide}
[z,w[x,y]] = [[w,y],xz] -[z,[y,wx]] + [x,[w,zy]] + x[z,w]y +[w,z]yx 
\end{equation}
which is checked by a direct calculation. 

Setting in this identity $w=z^{s-1}$, we get that $[z,z^{s-1}[x,y]] \in L_3(A_3(\ZZ))$. Now replacing $y$ with $x^{q-1}y^r$, we get that $[z,z^{s-1}[x,x^{q-1}y^r]] \in L_3(A_3(\ZZ))$. Setting $u:=[z,z^{s-1}x^{q-1}y^{r-1}[xy]]$, and using Lemma \ref{lem:FSLemma1}, we get that $ru \in L_3(A_3(\ZZ))$.  Similarly, $[z,z^{s-1}[y,x^qy^{r-1}]] \in L_3(A_3(\ZZ))$, and using Lemma \ref{lem:FSLemma1}, $qu \in L_3(A_3(\ZZ))$.

It remains to show that $su \in L_3(A_3(\ZZ))$. To this end, we set $m=q-1$, $k=r-1$, and write using Lemma \ref{lem:FSLemma1} (mod $L_3(A_3(\ZZ))$):
\begin{align*}
[z,x^my^k[xy]]&=[z,x^m[y^kx,y]] = [x^m,z[y^kx,y]] =[x^m,[zy]y^kx] \\
&=[x^m,[zy^k,y]x] =-[x^m,x[y,zy^k]] =-[x,x^m[y,zy^k]],
\end{align*}
which is in $L_3(A_3(\ZZ))$ by identity (\ref{ide}). 
So, putting $z^s$ instead of $z$ we get that $su=0$, as desired.

\item[\begin{small}2) \end{small}] We consider first the case when $\gcd(s, q, r) = 2$. Let $i = s-1, j = q-1, k = r -1$, and consider the element $F(x,y,z):=[x,x^iy^jz^k[yz]] = T(s, q, r)$. It suffices to show that this element is nontrivial in $B_2(A_3(\mathbb{F}_2))$. 

Let $t_x, t_y, t_z$ be commutative variables (i.e. they commute with $x,y,z$ and each other). Consider $F(x+t_x,y+t_y,z+t_z)$, and take the coefficient of tridegree  \linebreak $(i-1,j-1,k-1)$ in $t_x,t_y,t_z$. Clearly, it is $ijk[x,xyz[yz]]$. But we know, from computer calculations in MAGMA, that this is nonzero in $B_2(A_3(\mathbb{F}_2))$. We see that for all odd $i,j,k$,  $F(x,y,z)$ is nonzero in $B_2(A_3(\mathbb{F}_2))$.

A similar procedure works in the case where $\gcd(s,q,r)=3$, with $i,j,k$ of the form \ $3m-1$. The relevant coefficient (of degree $i-2,j-2,k-2$) would be $\binom{i}{2}\binom{j}{2}\binom{k}{2}[x,x^2y^2z^2[yz]]$, and the element $[x,x^2y^2z^2[yz]]$ is nonzero in $B_2(A_3(\mathbb{F}_3))$ by a computer calculation.

\end{enumerate}
\end{proof}
\begin{rem}
1. Actually, the argument in the proof for part 2 would work for any $\gcd(s,q,r) = m$ 
if we knew that $[x,x^{m-1}y^{m-1}z^{m-1}[yz]]$ has order exactly 
$m$ in $B_2(A_3(\ZZ))$.

2. Below we will give another proof that the order of $T(s,q,r)$ is exactly ${\rm gcd}(s,q,r)$, which works
when ${\rm gcd}(s,q,r)$ is odd. 
\end{rem}

\subsection{Torsion in $B_2(A_2(\Bbb Z))$}

\begin{thm} \label{noTorsionB2}
$B_2(A_2(\ZZ))$ has no torsion. 
\end{thm}

\begin{proof}  We make use of the following:
\begin{lem} \label{torLemma1}
We have $B_2(A_n(\ZZ)) = \sum_{i=1}^n [x_i,\bar{B}_1(A_n(\ZZ))]$. 
\end{lem}
\begin{proof}
The statement follows from Lemma \ref{degg1}.
\end{proof}

Let us denote the generators of $A_2$ by $x,y$. 

\begin{lem}  \label{torLemma2}
If $T$ is a torsion element of $\bar{B}_1(A_2(\ZZ))$, then $[x,T]=[y,T]=0$ in $B_2(A_2(\ZZ))$.
\end{lem}

\begin{proof}
By Corollary \ref{b1a2}, torsion elements are linear combinations of
elements of the form $T=x^{q-1}y^{r-1}[xy]$ (corresponding to the
2-form $x^{q-1}dx \wedge y^{r-1}dy$). Now, 
$[x,T]=[x,x^{q-1}y^{r-1}[xy]]$. This is the specialization of
$T(1,q,r)$ under setting $z=x$ (where by ``specialization", we mean
that we apply the homomorphism $A_3 \rightarrow A_2$ such that $x,y,z$
go to $x,y,x$, respectively). Since $T(1,q,r)=0$ in $B_2$, we
conclude that the specialization is zero as well, i.e., $[x,T]=0$. 
Similarly, $[y,T]=0$.
\end{proof}

We showed in Theorem \ref{thm:Thm1}, that \[ \bar{B}_1(A_2) =  \Omega^0\oplus\Omega^{2}/\Omega^{2}_{ex}. \]

Because the last summand is all torsion, by Lemma \ref{torLemma2}, we
can strengthen the formula for $B_2$ in Lemma \ref{torLemma1} to say
that $B_2=[x,\Omega^0]+[y,\Omega^0]$. Let us denote the $\ZZ$-span of $x$ and $y$ by $V$. Since $x,y$ is a basis of $V$,
we see that $B_2=[V,\Omega_0]$, i.e. $B_2$ is a quotient of $V\otimes
\Omega_0$. 

Note that we have an identification $V\otimes \Omega^0\cong \Omega^1$, via $x\otimes f+y\otimes g\mapsto fdx+gdy$.  Let us show that closed 1-forms map to zero in $B_2$.  
Let $f,g \in \Omega^0$ be such that $fdx+gdy$ is a closed form, i.e. $f_y=g_x$. We may assume that this form is homogeneous of bidegree $(q, r)$. Then we can set $f=qx^{q-1}y^r/d, g=rx^qy^{r-1}/d$, where $d=\gcd(q, r)$. 
Define lifts $\hat f$ and $\hat g$ of $f,g$ to $A_2$
by the formulas
$$
\hat f=\sum_{i=0}^{\frac{q}{d}-1}x^iy^{\frac{r}{d}}
(x^{\frac{q}{d}}y^{\frac{r}{d}})^{d-1}x^{\frac{q}{d}-1-i},
$$
$$
\hat g=\sum_{i=0}^{\frac{r}{d}-1}y^i
(x^{\frac{q}{d}}y^{\frac{r}{d}})^{d-1}x^{\frac{q}{d}}y^{\frac{r}{d}-1-i}.
$$
Then it is easy to see that 
\[ 
[x,\hat f]+[y,\hat g]=0. 
\]
in $A_2$. This implies that the closed 1-form $fdx+gdy$ is killed, as desired.

Thus, we see that $B_2(A_2(\ZZ))$ is a quotient of the 
free group $\Omega^1/\Omega^1_{cl}$. But by the Feigin-Shoikhet theorem \cite{FS}, 
this free group already has the same Hilbert series as $B_2(A_2(\mathbb{Q}))$. So, there is no torsion.
\end{proof}

\subsection{Torsion in $B_2(A_n(\Bbb Z))$}

\begin{conj}\label{torb2}
The graded abelian group $B_2(A_n(\Bbb Z))$ 
is isomorphic to 
$$
\oplus_{i\ge 1}\Omega^{2i+1}(\Bbb Z^n)/\Omega^{2i+1}_{ex}(\Bbb Z^n).
$$
Thus, the torsion in $B_2(A_n(\ZZ))$ is isomorphic to 
\[ H^3(\Omega(\ZZ^n)) \oplus H^5(\Omega(\ZZ^n)) \ldots = 
\bigoplus^{\lfloor{\frac{n-1}{2}\rfloor}} _{k=1} H^{2k+1}(\Omega(\ZZ^n)). \]
In particular, the torsion of $B_2(A_3(\ZZ))[q,r,s]$ is spanned by the element $T(s, q, r)$, which has order exactly $\gcd(s, q, r)$. 
\end{conj}

This conjecture implies the following conjecture in characteristic $p$:

\begin{conj}\label{b2charp}
If all $m_i$ are positive, then $\dim B_2(A_n(\Bbb F_p))[m_1,...,m_n]$
is $2^{n-2}$ if there exists $i$ such that $m_i$ is not divisible by $p$, and $2^{n-1}-1$ 
otherwise (i.e., if all $m_i$ are divisible by $p$). 
\end{conj}

The following theorem shows that Conjectures \ref{torb2} and \ref{b2charp} hold at least as upper bounds. 

\begin{thm}\label{upbo}
The torsion in $B_2(A_n(\ZZ))$ is a quotient of
\[ H^3(\Omega(\ZZ^n)) \oplus H^5(\Omega(\ZZ^n)) \ldots = \bigoplus^{\lfloor{\frac{n-1}{2}\rfloor}} _{k=1} H^{2k+1}(\Omega(\ZZ^n)). \]
In particular, the torsion of $B_2(A_3(\ZZ))[q,r,p]$ is spanned by the element $T(p, q, r)$. 
Also, the numbers in Conjecture \ref{b2charp} are upper bounds for 
$\dim B_2(A_n(\Bbb F_p))[m_1,...,m_n]$. 
\end{thm}

\begin{proof} 
Let $V$ be the $\Bbb Z$-span of the generators $x_1,...,x_n$. 
By Lemma \ref{torLemma1} and Corollary \ref{deco}, we have a surjective   
homomorphism $\xi: V\otimes \Omega^{ev}(\Bbb Z^n)/\Omega^{ev}_{ex}(\Bbb Z^n)
\to B_2(A_n(\Bbb Z))$. 
Moreover, for any $a,b,c,z\in A_n$, we have 
$$
[a,[b,c]z]+[b,[a,c]z]=[a[b,c],z]-[[b,c],za]+[b[a,c],z]-[[a,c],zb]=
$$
$$
=[[ab,c],z]-[[b,c],za]+[[b,[a,c]],z]-[[a,c],zb]
\in L_3,
$$
which shows that the image of $[a,[b,c]z]$ in $B_2$ 
is skewsymmetric in $a,b,c$. This implies that 
$\xi$ is a composition of a homomorphism 
$$
\eta: \Omega^1(\Bbb Z^n)\oplus \Omega^{odd,\ge 3}(\Bbb Z^n)/\Omega^{odd,\ge 3}_{ex}(\Bbb Z^n)\to B_2(A_n(\Bbb Z))
$$ 
and the map 
$v\otimes \omega\mapsto dv\wedge \omega$. 
So the theorem follows from the following lemma. 

\begin{lem}\label{le}
The image $\tilde I$ 
of $\Omega^1_{cl}(\Bbb Z^n)\oplus \Omega^{odd,\ge 3}(\Bbb Z^n)/\Omega^{odd,\ge 3}_{ex}(\Bbb Z^n)$ under $\eta$ 
coincides with the image $I$ of $\Omega^{odd,\ge 3}(\Bbb Z^n)/\Omega^{odd,\ge 3}_{ex}(\Bbb Z^n)$
under $\eta$.  
\end{lem} 

Indeed, by the results of \cite{FS}, the natural map 
$\gamma: \Omega^1(\ZZ^n)/\Omega^1_{cl}(\ZZ^n)\to B_2(A_n(\Bbb Z))/\tilde I$ is an isomorphism 
after tensoring with $\Bbb Q$. Since the group
$\Omega^1(\ZZ^n)/\Omega^1_{cl}(\ZZ^n)$ is free, $\gamma$ is injective. Hence,  
Lemma \ref{le} shows that we have an exact sequence
$$
0\to K\to \Omega^{odd,\ge 3}(\Bbb Z^n)/\Omega^{odd,\ge 3}_{ex}(\Bbb Z^n)\to B_2(A_n(\Bbb Z))\to 
\Omega^1(\Bbb Z^n)/\Omega^1_{cl}(\Bbb Z^n)\to 0,
$$
where the last nontrivial map is induced by $\gamma^{-1}$. 
So, since $\Omega^1/\Omega^1_{cl}$ is a free group, we have 
$$
B_2(A_n(\Bbb Z))\cong \Omega^1(\Bbb Z^n)/\Omega^1_{cl}(\Bbb Z^n)\oplus 
(\Omega^{odd,\ge 3}(\Bbb Z^n)/\Omega^{odd,\ge 3}_{ex}(\Bbb Z^n))/K.
$$
Moreover, as follows from the Feigin-Shoikhet results \cite{FS}, this isomorphism 
holds without $K$ after tensoring with $\Bbb Q$, which implies that 
$K$ is a torsion group. So we get 
$$
{\rm tor}_2B_2(A_n(\Bbb Z))\cong 
(\Omega^{odd,\ge 3}_{cl}(\Bbb Z^n)/\Omega^{odd,\ge 3}_{ex}(\Bbb Z^n))/K=
(H^3\oplus H^5\oplus...)/K,
$$
as desired. 

\begin{proof} (of Lemma \ref{le}) The proof is similar to the argument at the end of 
the proof of Theorem \ref{noTorsionB2}.
Namely, let $f_1dx_1+...+f_ndx_n$ be a closed 1-form, which is homogeneous of multidegree $m_1,...,m_n$.
We may assume that $m_i>0$, otherwise we are reduced to smaller $n$.  
Let $d={\rm gcd}(m_1,...,m_n)$. 
Then we can set $f_i=\frac{m_i}{d}x_i^{-1}\prod_j x_j^{m_j}$. 
Define lifts $\hat f_i$ of $f_i$ to elements in $A_n$ by the formulas
$$
\hat f_i=\sum_{j=0}^{\frac{m_i}{d}-1}x_i^jx_{i+1}^{\frac{m_{i+1}}{d}}\dots x_n^{\frac{m_n}{d}}
(x_1^{\frac{m_1}{d}}\dots x_n^{\frac{m_n}{d}})^{d-1}
x_1^{\frac{m_1}{d}}\dots x_{i-1}^{\frac{m_{i-1}}{d}}x_i^{\frac{m_i}{d}-j-1}. 
$$
It is easy to see that $\sum_i [x_i,f_i]=0$. 
This implies that $\eta(\sum_i f_idx_i)\in I$, 
as desired. 
\end{proof}

\end{proof}

\begin{rem}
Theorem \ref{upbo} implies that Conjecture \ref{b2charp} holds 
for $p=2,3$ and $n\le 4$. The proof is analogous to the proof of Theorem 
\ref{TorsionTheorem}(2).  
\end{rem}

\subsection{Proof of 2-localized versions of Conjectures \ref{torb2} and \ref{b2charp}}

\subsubsection{The statement}

\begin{thm}\label{maint}
(i) Conjecture \ref{torb2} holds over $\Bbb Z[1/2]$. 
So if ${\rm gcd}(q,r,s)=2^\ell(2k+1)$ then 
the order of $T(q,r,s)$ is divisible by $2k+1$. 

(ii) Conjecture \ref{b2charp} holds for $p>2$. 
\end{thm}

The rest of the section is the proof of Theorem \ref{maint}.
In view of Corollary \ref{unco} (2), it suffices to prove (i).  
We will apply the theory from the appendix to \cite{DKM}, namely Theorem 7.2.

\subsubsection{First cyclic homology of the free algebra over the integers} 

Let $A=A_n(\Bbb Z)$. First of all we compute 
the first cyclic homology group $HC_1(A)$. 
This computation is known (due to Loday and Quillen), but we will give it 
here for the convenience of the reader. 

Recall that $HC_1(A)$ is the quotient of the kernel of the commutator map
$[,]:\wedge^2A\to A$ by the elements $ab\wedge c+bc\wedge a+ca\wedge b$, 
$a,b,c\in A$. Also recall that if $a=x_{i_1}...x_{i_N}$ is a cyclic 
word in $A/[A,A]$ (i.e., a word up to cyclic permutation) 
then we can define its noncommutative partial derivative $\partial_i a\in A$ 
by the formula 
$$
\partial_ia:=\sum_{k: i_k=i}x_{i_{k+1}}...x_{i_N}
x_{i_1}...x_{i_{k-1}}  
$$
Finally, note that for any cyclic word $a$ and any positive 
integer $m$, the partial derivative $\partial_i(a^m/m)$, 
has integer coefficients, even though $a^m/m$ does not. 

\begin{thm}\label{cychom} (\cite{LQ}, Proposition 5.4)
If $a$ is a non-power word (i.e., a nontrivial 
word that is not a power of another word) 
and $m>1$ is a positive integer, then the noncommutative differential 
$$
d(a^m/m):=\sum_i \partial_i(a^m/m)\otimes x_i 
$$
defines 
an element of $HC_1(A)$ of order $m$. Moreover, 
$HC_1(A)$ is the direct sum of the cyclic 
subgroups $\Bbb Z/m$ spanned by these elements.  
\end{thm} 

\begin{proof} 
We have the Connes-Loday-Quillen exact sequence (\cite{LQ}, Theorem 1.6)
$$
HC_0(A)=A/[A,A]\to HH_1(A)\to HC_1(A)\to 0.
$$
Now, since $A=TV$ is a free algebra, 
$HH_1(A)$ is the kernel of the commutator map $A\otimes V\to A$,
and the map $d: HC_0(A)\to HH_1(A)$ is
the noncommutative differential. 
In characteristic zero, the kernel of the map $d$ is just constants, 
and  $HC_1=0$. This implies that over $\Bbb Z$, $HC_1$ is a torsion group, 
and in each multidegree $\bold m$, one has 
$$
HC_1(A)[\bold m]=\hat J_{\bold m}/J_{\bold m},
$$
where $J_{\bold m}:={\rm Im}(d)[\bold m]$, and  
$\hat J_{\bold m}$ is the saturation of $J_{\bold m}$ 
inside $HH_1(A)$. Now, $\hat J_{\bold m}/J_{\bold m}$ 
is clearly a direct sum of groups $\Bbb Z/m$ generated by 
elements $d(a^m/m)$, where $a$ is a non-power word, 
which implies the statement. 
\end{proof} 

\subsubsection{Conclusion of proof of Theorem \ref{maint}} 

Now we are ready to prove Theorem \ref{maint}. 
We will work over $\Bbb Z[1/2]$ and for brevity will not 
explicitly show it in the notation. 
Recall from the appendix to \cite{DKM}, proof 
of Theorem 7.2, that we have an exact sequence 
\begin{equation}\label{exase}
HC_1(A)\to \wedge^2(A/(L_2+M_3))/W\to B_2(A)\to 0,
\end{equation}
where $W$ is spanned by the images of the 
elements $ab\wedge c+bc\wedge a+ca\wedge b$,
$a,b,c\in A$. 

Now recall that $\Omega^{ev}$ is equipped with 
the Fedosov product $a*b=ab+\frac{1}{2}da\wedge db$, 
and analogously to \cite{FS}, by Proposition \ref{phimap} 
we have an algebra isomorphism 
$\phi: A/M_3\to \Omega^{ev}_*$, 
which maps $A/(L_2+M_3)=\bar B_1(A)$ isomorphically onto 
$\Omega^{ev}/\Omega^{ev}_{ex}$.   
So, the middle term of the sequence (\ref{exase}) is 
$\wedge^2(\Omega^{ev}/\Omega^{ev}_{ex})/W$,
where $W$ is now spanned by the elements $ab\wedge c+bc\wedge a+ca\wedge b$,
$a,b,c\in \Omega^{ev}/\Omega^{ev}_{ex}$ (note that it is not important 
whether $ab$ is the usual or Fedosov product, as they differ by an exact form, 
and we are working modulo exact forms). 

Now, it is shown as in \cite{DKM}, Theorem 7.1, 
that the algebra of differential forms is pseudoregular in the sense of 
\cite{DKM}, so as a result the map 
$$
\theta: \wedge^2(\Omega^{ev}/\Omega^{ev}_{ex})/W\to \Omega^{odd}/\Omega^{odd}_{ex}
$$ 
given by $\theta(a\otimes b)=a\wedge db$ is an isomorphism. 
Thus, we have an exact sequence 
$$
HC_1(A)\to \Omega^{odd}/\Omega^{odd}_{ex}\to B_2(A)\to 0, 
$$
where the first map is defined by the formula 
$$
T=\sum_i f_i\otimes x_i\mapsto \sum_i \phi(f_i)\wedge dx_i.
$$
Our job is to show that this map lands in $\Omega^1/\Omega^1_{ex}$.

By Theorem \ref{cychom}, for this it suffices to show that 
for every cyclic word $a$ of multidegree $(m_1,...,m_n)$, 
$$
\sum_i \phi(\partial_i(a^m/m))\wedge dx_i\in \Omega^1+\Omega^{odd,\ge 3}_{ex}.
$$ 
This is equivalent to saying that the form
$$
\omega(a,m):=\sum_i (\phi(\partial_i(a^m/m))-\partial_i(a^m/m))\wedge dx_i
$$ 
belongs to $ \Omega^{odd,\ge 3}_{ex}$.
Since this form is in $\Omega^{\ge 3}$ and is clearly closed, 
it suffices to show that it represents the trivial class in the 
De Rham cohomology. But by Corollary \ref{Hev-description}, 
for this it suffices to show that $\omega(a,m)$ is divisible by $mD$, where 
$D={\rm gcd}(m_1,...,m_n)$. 

To this end, let us compute $\omega(a,m)$ explicitly. 
Suppose that $w$ be a cyclic word. By a shuffle subword of $w$ we will mean 
a cyclic word obtained by crossing out some letters from $w$. 
Let $1\le i_1<i_2<...<i_{2k+1}\le n$.  
Let $N_{i_1,...,i_{2k+1}}(w)$ be the number of 
shuffle subwords of $w$ which are even permutations of 
$x_{i_1},...,x_{i_{2k+1}}$ minus the number of shuffle subwords 
which are odd permutations. Note that since the length of the shuffle subword is odd, 
it makes sense to talk about its parity, as it does not change under cyclic permutation.

\begin{lem}\label{form}
One has 
$$
\sum_i \phi(\partial_iw)\wedge dx_i=\sum_{k\ge 0}\frac{1}{2^k}\sum_{i_1<...<i_{2k+1}}N_{i_1,...,i_{2k+1}}(w)dx_{i_1}\wedge...\wedge dx_{i_{2k+1}}.
$$
\end{lem}

\begin{proof}
This follows by direct calculation using the formula for the Fedosov product.  
\end{proof}

\begin{cor}
One has 
$$
\omega(a,m)=\frac{1}{m}\sum_{k\ge 1}\frac{1}{2^k}\sum_{i_1<...<i_{2k+1}}
N_{i_1,...,i_{2k+1}}(a^m)dx_{i_1}\wedge...\wedge dx_{i_{2k+1}}.
$$
\end{cor} 

Therefore, the theorem follows from 
the following combinatorial lemma. 

\begin{lem}\label{comb} 
Let $a$ be a cyclic word 
of degrees $m_i$ with respect to $x_i$, and $D={\rm gcd}(m_1,...,m_n)$. 
Then for $k\ge 1$, the number 
$N_{i_1,...,i_{2k+1}}(a^m)$ is divisible by $m^{k+1}D$.  
\end{lem}

\begin{proof}
Let $y_i$, $i=1,...,n$ 
be anticommuting variables (i.e., $y_iy_j=-y_jy_i$ and $y_i^2=0$).
Suppose that $a=x_{j_1}...x_{j_M}$ (where $M=\sum m_i$), 
and consider the product $Y(a,m)=((1+y_{j_1})...(1+y_{j_M}))^m$. 
It is easy to see that $N_{i_1,...,i_{2k+1}}(a^m)$ 
is the coefficient of $y_{i_1}...y_{i_{2k+1}}$ in $Y(a,m)$. 
However, it is easily shown by induction that 
$$
(1+y_{j_1})...(1+y_{j_M})=(1+m_1y_1+...+m_ny_n)\prod_{1\le r<s\le n}(1+y_{j_r}y_{j_s}). 
$$
This implies that 
$$
Y(a,m)=(1+mm_1y_1+...+mm_ny_n)\prod_{1\le r<s\le n}(1+my_{j_r}y_{j_s})
$$
and the statement follows. 
\end{proof} 

The theorem is proved.  

\section{Experimental data and conjectures}

\subsection{Experimental data}
\label{exp-data-I}
In this subsection we summarize the experimental data obtained by direct computation in MAGMA.

In the tables below, multi-degree components are greater than or equal to $1$ since cases with one degree being $0$ reduce to a smaller number of variables.  Due to the $S_n$ action permuting generators, we list only multi-degrees with weakly descending entries.  Note that we are asserting the torsion subgroups are trivial for omitted degrees in each specified range.



\begin{table}[htpb]
\caption{Torsion in $B_\ell(A_2(\ZZ))$ in degrees $(i,j)$, with $i\geq j$, such that $2 \le \ell \le i+j \le 12$.}
\begin{center}
    \begin{tabular}{ | l | | l | l | l | l |}
    \hline
     $(i,j) \backslash \ell$: & 5 & 6 & 7 & 9\\ \hline
     (4,4) &$\ZZ/2$&&&\\  \hline
     (6,4) &$\ZZ/2$&&$\ZZ/2$&\\ \hline
     (6,6) &$\ZZ/2$&$\ZZ/3 $ & $(\ZZ/2)^3$& $(\ZZ/2)^3$\\  \hline
     (8,4) &$\ZZ/2$ && $\ZZ/2$&$(\ZZ/2)^2$\\  \hline
     \end{tabular}
\end{center}
\label{torsiontwo}
\end{table}

\begin{rem}
Note that we did not discover torsion in $B_2(A_2(\ZZ))$. This lack of torsion is proved in Theorem \ref{noTorsionB2}.
\end{rem}



\begin{table}[htpb]
\caption{Torsion in $B_\ell(A_3(\ZZ))$ in degrees $(i,j,k)$, with $i\geq j \geq k$, such that $2 \le \ell \le i+j+k$ and either $i,j,k \le 3$ or $j,k \le 4$, $i \le 2$.}
\begin{center}
    \begin{tabular}{ | l | | l | l | l | l | l |}
    \hline
      $(i,j,k) \backslash \ell$: & 2 & 3 & 5 & 7 \\ \hline
     (2,2,2) & $\ZZ/2$&&& \\ \hline
     (3,3,3) &$\ZZ/3 $ & $\ZZ/3$&&\\ \hline
    (4,2,2) &$\ZZ/2$ && $(\ZZ/2)^2$& \\ \hline
     (4,4,2) &$ \ZZ/2$&&$(\ZZ/2)^5$&$(\ZZ/2)^5$ \\ \hline
      \end{tabular}
 \end{center}
 \label{torsionthree}
 \end{table}
 
 

\begin{table}[htpb]
\caption{Torsion in $B_\ell(A_4(\ZZ))$ in degrees $(i,j,k,l)$, with $i\geq j \geq k \geq l$, such that $2 \le \ell \le i+j+k+l$ and $i,j,k,l \le 2$.}
\begin{center}
    \begin{tabular}{ | l | | l |  l |}
    \hline
     $(i,j,k,l) \backslash \ell$ & 2 &5 \\ \hline
     (2,2,2,2) &  $(\ZZ/2)^3$ & $(\ZZ/2)^5$ \\ \hline 
      \end{tabular}
 \end{center}
  \label{torsionfour}
 \end{table}

\subsubsection{Discrepancies between ranks of $B_\ell$ in characteristic zero and positive characteristic}\label{discr}

The difference
$D=\dim B_\ell(A_3(\mathbb{\Bbb F}_2))-\dim B_\ell(A_3(\mathbb{Q}))$ in degrees $(i,j,k)$ 
was calculated using MAGMA in degrees up to $(3,3,3)$ and $(4,2,2)$ 
and was found to be nonzero in the following cases:

\[
\begin{array}{|c|c|c|} \hline
\ell & (i,j,k) & D \\ \hline \hline
2 & (2, 2, 2) &1 \\ \hline
2 & (4, 2, 2) &1 \\ \hline
4 & (2, 2, 2) &-1 \\ \hline
4 & (4, 2, 2) &-1 \\ \hline
5 & (4, 2, 2) &2 \\ \hline
6 & (4, 2, 2) &-2 \\ \hline
\end{array}
\]

The difference
$D=\dim B_\ell(A_3(\mathbb{\Bbb F}_3))-\dim B_\ell(A_3(\mathbb{Q}))$ in degrees $(i,j,k)$ 
calculated using MAGMA in degrees up to $(3,3,3)$ and $(4,2,2)$ 
and was found to be nonzero in the following cases:

\[
\begin{array}{|c|c|c|} \hline
\ell & (i,j,k) & D \\ \hline \hline
2 & (3, 3, 3)& 1 \\ \hline
5 & (3, 3, 3) &-1 \\ \hline
\end{array}
\]
 
\subsection{Conjectures}
 
The theoretical results of this paper and the computational results in the previous subsection motivate the following conjectures.
 
\begin{conj}
If a torsion element $T$ in $B_\ell(A_n(\ZZ))$ has degree $m$ with respect to some generator $x_j$, then the order of $T$ divides $m$.
\end{conj}
 
\begin{conj}
There is no torsion in $B_\ell(A_n(\ZZ))[m_1, \ldots , m_n]$ unless $\ell \le m_1+ \cdots + m_n-3$.
\end{conj}
 
\begin{rem}\label{weakerst}
One can prove a weaker statement that there is no torsion unless $ \ell \le m_1+ \cdots + m_n-2$. Indeed, it is clear that $B_\ell[m_1,...,m_n]$ has no torsion if $m_1+...+m_n=\ell$ because $L_{\ell+1}[m_1,...,m_n]=0$ in this case. Also, in this case $L_\ell[m_1,...,m_n]$ is a saturated subgroup in the free algebra $A_n(\ZZ)$ (i.e., if $jx\in L_\ell$  for a positive integer $j$, then $x\in L_\ell$) since this is a graded component of the free Lie algebra in $n$ generators over $\ZZ$ inside its universal enveloping algebra, which is $A_n(\ZZ)$. This means that $B_{\ell-1}[m_1,...,m_n]$ is also free in this case. 
 \end{rem}

Also we propose the following questions: 

{\bf Question 1.} Suppose that $\ell>2$ and $n$ are fixed. 
Can there exist $p$-torsion 
in $B_\ell(A_n(\Bbb Z))$ for arbitrarily large $p$? 
What if $\ell,n$ are allowed to vary? 
(In the computer search, we only found $2$-torsion 
and $3$-torsion). 
 
{\bf Question 2.} 
Does $2$-torsion in $B_\ell(A_2(\ZZ))$ where $\ell>2$ occur only for odd $\ell$?

\section{The supercase}\label{supercase}

In this section we consider the super-extension of the preceding results, in the 
style of the paper \cite{BJ}. Namely, let $A_{n,k}(R)$ 
be the free algebra generated over a commutative ring $R$ 
by $n$ even generators $x_1,...,x_n$ and $k$ odd generators, $y_1,...,y_k$. 
This is really the same algebra as $A_{n+k}(R)$; the only thing that changes is the notion of 
the commutator. Namely, for two words $a,b\in A_{n,k}$, 
we define $[a,b]=ab-(-1)^{d_ad_b}ba$, where $d_a$ is the number of odd 
generators $y_j$ in the word $a$. Then we define the modules
$L_i$ and $B_i$ in the same way as in the usual case. 

The structure of $\bar B_1(A_{n,k})$ and $B_i(A_{n,k})$ over $\Bbb Q$ was studied in \cite{BJ}. 
In particular, the structure of $\bar B_1$, $B_2$, and $B_3$ was completely computed. 
Here we provide some results and data and formulate a number of questions regarding the structure of
$\bar B_1(A_{n,k})$ and $B_i(A_{n,k})$ over $\Bbb Z$. 

\subsection{Experimental data}

Tables 4-8 list the torsion subgroup of $B_\ell(A_{m,n})(\ZZ)$ for small values of $\ell,m,n$.  Note that, as in Section \ref{exp-data-I}, we are asserting these subgroups to be trivial for all omitted degrees in the described range.


\begin{table}[h!]
\caption{Torsion in $B_\ell(A_{1,1}(\mathbb Z))$ in degrees $(r,s)$, where $r$ is the degree of the even generator and $s$ is the degree of the odd generator, such that $2 \leq \ell \leq r+s \leq 11$, $r,s\leq 9$, excepting $(r,s)=(8,3)$.}
$ \begin{array}{ |c||c|c|c|c|c|c|c|c |  } \hline
   (r,s) \backslash \ell &2&3&4&5&6&7&8&9 \\ \hline
(2,2)& \ZZ/ 2 &&&&&&& \\ \hline
(3,3)&&&\ZZ/2&&&&& \\ \hline
(4,2)&\ZZ/2&&\ZZ/2&&&&& \\ \hline
(3,4)&&&\ZZ/2&\ZZ/2&&&& \\ \hline
(4,3)&&&\ZZ/2&\ZZ/2&&&& \\ \hline
(2,6)&\ZZ/2&&\ZZ/2&&\ZZ/2&&&\\ \hline
(3,5)&&&\ZZ/2&\ZZ/2&(\ZZ/2)^2&&& \\ \hline
(4,4)&\ZZ/2&&\ZZ/2&(\ZZ/2)^2&(\ZZ/2)^3&&& \\ \hline
(5,3)&&&\ZZ/2&\ZZ/2&(\ZZ/2)^2&&& \\ \hline
(6,2)&\ZZ/2&&\ZZ/2&&\ZZ/2&&& \\ \hline
(3,6)&&&\ZZ/2&\ZZ/2&(\ZZ/2)^3&(\ZZ/2)^2&&\\ \hline
(4,5)&&&\ZZ/2&(\ZZ/2)^2&(\ZZ/2)^7&(\ZZ/2)^4&& \\ \hline
(5,4)&&&\ZZ/2&(\ZZ/2)^2&(\ZZ/2)^6&(\ZZ/2)^4&& \\ \hline
(6,3)&&&\ZZ/2&\ZZ/2&(\ZZ/2)^2&\ZZ/2&& \\ \hline
(3,7)&&&\ZZ/2&\ZZ/2&(\ZZ/2)^3&(\ZZ/2)^3&(\ZZ/2)^3& \\ \hline
(4,6)&\ZZ/2&&(\ZZ/2)^2&(\ZZ/2)^2&(\ZZ/2)^{10}&(\ZZ/2)^{10}&(\ZZ/2)^8 &\\ \hline
(5,5)&&&\ZZ/2&(\ZZ/2)^2&(\ZZ/2)^{11}&(\ZZ/2)^{13}&(\ZZ/2)^9& \\ \hline
(6,4)&&&\ZZ/2&(\ZZ/2)^2&(\ZZ/2)^6&(\ZZ/2)^8&(\ZZ/2)^6& \\  \hline
(7,3)&&&\ZZ/2&\ZZ/2&(\ZZ/2)^2&(\ZZ/2)^2&(\ZZ/2)^3&\\ \hline
(8,2)&\ZZ/2&&\ZZ/2&&\ZZ/2&&\ZZ/2 &\\ \hline
(3,8)&&&\ZZ/2&\ZZ/2&(\ZZ/2)^3&(\ZZ/2)^3&(\ZZ/2)^5&(\ZZ/2)^3\\ \hline
(4,7)&&&\ZZ/2&(\ZZ/2)^2&(\ZZ/2)^8&(\ZZ/2)^{11}&(\ZZ/2)^{19}&(\ZZ/2)^9\\ \hline
(5,6)&&&\ZZ/2&(\ZZ/2)^2&(\ZZ/2)^{12}&(\ZZ/2)^{20}&(\ZZ/2)^{31}&(\ZZ/2)^{14}\\ \hline
(6,5)&&&\ZZ/2&(\ZZ/2)^2&(\ZZ/2)^{11}&(\ZZ/2)^{18}&(\ZZ/2)^{27}&(\ZZ/2)^{13}\\ \hline
(7,4)&&&\ZZ/2&(\ZZ/2)^2&(\ZZ/2)^6&(\ZZ/2)^8&(\ZZ/2)^{13}&(\ZZ/2)^8\\\hline
     \end{array}$
\end{table}

   
\begin{table}[h!]
\caption{Torsion in $B_\ell(A_{0,2}(\mathbb Z))$ in degrees $(s_1,s_2)$, $s_1\ge s_2$, such that $2 \leq \ell \leq s_1+s_2 \leq 11$, $2\leq s_1,s_2\leq 9$.}
$ \begin{array}{ |c||c|c|c|c|c|c|c|c|  } \hline
     (s_1,s_2) \backslash \ell &2&3&4&5&6&7&8&9 \\ \hline
(3,3)&&&\ZZ/2&&&&& \\ \hline
(4,2)&\ZZ/2&&\ZZ/2&&&&& \\ \hline
(4,3)&&&\ZZ/2&\ZZ/2&&&& \\ \hline
(4,4)&&&\ZZ/2&(\ZZ/2)^2&(\ZZ/2)^3&&& \\ \hline
(5,3)&&&\ZZ/2&\ZZ/2&(\ZZ/2)^2&&& \\ \hline
(5,4)&&&\ZZ/2&(\ZZ/2)^2&(\ZZ/2)^7&(\ZZ/2)^4&& \\ \hline
(6,3)&&&\ZZ/2&\ZZ/2&(\ZZ/2)^3&(\ZZ/2)^2&& \\ \hline
(5,5)&&&\ZZ/2&(\ZZ/2)^2&(\ZZ/2)^{12}&(\ZZ/2)^{14}&(\ZZ/2)^{10}& \\ \hline
(6,4)&\ZZ/2&&(\ZZ/2)^2&(\ZZ/2)^2&(\ZZ/2)^{10}&(\ZZ/2)^{10}&(\ZZ/2)^8& \\ \hline
(7,3)&&&\ZZ/2&\ZZ/2&(\ZZ/2)^3&(\ZZ/2)^3&(\ZZ/2)^3& \\ \hline
(8,2)&\ZZ/2 &&\ZZ/2 && \ZZ/2 && \ZZ/2&\\ \hline
(6,5)&&&\ZZ/2&(\ZZ/2)^2&(\ZZ/2)^{12}&(\ZZ/2)^{21}& (\ZZ/2)^{33}&(\ZZ/2)^{15}\\ \hline
(7,4)&&&\ZZ/2&(\ZZ/2)^2&(\ZZ/2)^8&(\ZZ/2)^{11}&(\ZZ/2)^{19}&(\ZZ/2)^9\\ \hline
(8,3)&&&\ZZ/2&\ZZ/2&(\ZZ/2)^3&(\ZZ/2)^3&(\ZZ/2)^5&(\ZZ/2)^3\\ \hline

\end{array}
$\end{table}

\begin{table}[h!]
\caption{Torsion in $B_\ell(A_{2,1}(\mathbb Z))$ in degrees $(r_1,r_2,s)$, $r_1\ge r_2$, such that either $r_1,r_2,s\leq 3$ or $r_1\leq 4, r_2\leq 3, s\leq 2$.}
$\begin{array}{|c||c|c|c|c|c|c|} \hline
(r_1,r_2,s) \backslash \ell &2&3&4&5&6&7 \\ \hline
(2,1,3)&&&\ZZ/2&&& \\ \hline
(2,2,2)&(\ZZ/2)^2&&(\ZZ/2)^2&&& \\ \hline
(3,1,2)&&&\ZZ/2&&& \\ \hline
(2,2,3)&&&(\ZZ/2)^4&(\ZZ/2)^2&& \\ \hline
(3,1,3)&&&(\ZZ/2)^3&(\ZZ/2)^2&& \\ \hline
(3,2,2)&&&(\ZZ/2)^2&\ZZ/2&& \\ \hline
(4,1,2)&&&\ZZ/2&\ZZ/2&& \\ \hline
(3,2,3)&&&(\ZZ/2)^7&(\ZZ/2)^8&(\ZZ/2)^{10}& \\ \hline
(3,3,2)&&&(\ZZ/2)^4&(\ZZ/2)^3&(\ZZ/2)^5& \\ \hline
(4,2,2)&(\ZZ/2)^2&&(\ZZ/2)^4&(\ZZ/2)^3&(\ZZ/2)^5& \\ \hline
(3,3,3)&&\ZZ/3&(\ZZ/2)^{10}&(\ZZ/2)^{18}&(\ZZ/2)^{40}&(\ZZ/2)^{20}+\ZZ/3 \\ \hline
(4,3,2)&&&(\ZZ/2)^4&(\ZZ/2)^6&(\ZZ/2)^{12}&(\ZZ/2)^7 \\ \hline
\end{array}$
\end{table}


\begin{table}[h!]
\caption{torsion in $B_\ell(A_{1,2}(\mathbb Z))$ in degrees $(r,s_1,s_2)$, $s_1\ge s_2$, such that either $r,s_1,s_2\leq 3$ or $r \leq 4, s_1\leq 3, s_2 \leq 2$.}
$\begin{array}{|c||c|c|c|c|c|c|} \hline
(r,s_1,s_2) \backslash \ell &2&3&4&5&6&7 \\ \hline
(1,2,2)&&\ZZ/3&&&& \\ \hline
(1,3,2)&&\ZZ/3&\ZZ/2+\ZZ/3&&& \\ \hline
(2,2,2)&\ZZ/2&\ZZ/3&\ZZ/2+\ZZ/3&&& \\ \hline
(2,3,1)&&&\ZZ/2&&& \\ \hline
(3,2,1)&&&\ZZ/2&&& \\ \hline
(1,3,3)&&\ZZ/3&(\ZZ/2)^3+(\ZZ/3)^2&(\ZZ/2)^2+\ZZ/3&& \\ \hline
(2,3,2)&&\ZZ/3&(\ZZ/2)^4+(\ZZ/3)^2&(\ZZ/2)^2+\ZZ/3&& \\ \hline
(3,2,2)&&\ZZ/3&(\ZZ/2)^3+\ZZ/3&(\ZZ/2)^2+\ZZ/3&& \\ \hline
(3,3,1)&&&(\ZZ/2)^3&(\ZZ/2)^2&& \\ \hline
(4,2,1)&&&\ZZ/2&\ZZ/2&& \\ \hline
(2,3,3)&&\ZZ/3&(\ZZ/2)^8+(\ZZ/3)^3&(\ZZ/2)^{11}+(\ZZ/3)^3&(\ZZ/2)^{12}+(\ZZ/3)^3& \\ \hline
(3,3,2)&&\ZZ/3&(\ZZ/2)^7+(\ZZ/3)^2&(\ZZ/2)^{10}+(\ZZ/3)^2&(\ZZ/2)^{11}+(\ZZ/3)^2& \\ \hline
(4,2,2)&\ZZ/2&\ZZ/3&(\ZZ/2)^3+\ZZ/3&(\ZZ/2)^5+\ZZ/3&(\ZZ/2)^6+(\ZZ/3)^2& \\ \hline
(4,3,1)&&&(\ZZ/2)^3&(\ZZ/2)^5&(\ZZ/2)^6& \\ \hline
(3,3,3)&&&(\ZZ/2)^{11}+(\ZZ/3)^3&(\ZZ/2)^{25}+(\ZZ/3)^4&(\ZZ/2)^{60}+(\ZZ/3)^7&(\ZZ/2)^{28}+(\ZZ/3)^3 \\ \hline
(4,3,2)&&\ZZ/3&(\ZZ/2)^7+(\ZZ/3)^2&(\ZZ/2)^{15}+(\ZZ/3)^2&(\ZZ/2)^{38}+(\ZZ/3)^4&(\ZZ/2)^{19}+(\ZZ/3)^2 \\ \hline
\end{array}$
\end{table}

\begin{table}[h!]
\caption{Torsion in $B_\ell(A_{0,3}(\mathbb Z))$ in degrees $\mathbf{s}=(s_1,s_2,s_3)$, , such that either $2 \leq \ell \leq s_1+s_2+s_3\leq 10$ and  $s_1,s_2,s_3\leq 8$, or $\mathbf{s}=(7,3,1)$.}
{\tiny$
\begin{array}{|c||c|c|c|c|c|c|c|c|} \hline
\mathbf{s} \backslash \ell &2&3&4&5&6&7&8&9 \\ \hline
(2,2,1)&&\ZZ/3&&&&&& \\ \hline
(2,2,2)&(\ZZ/2)^2&(\ZZ/3)^2&(\ZZ/2)^2+(\ZZ/3)^2&&&&& \\ \hline
(3,2,1)&&\ZZ/3&\ZZ/2+\ZZ/3&&&&& \\ \hline
(3,2,2)&&\ZZ/3&(\ZZ/2)^4+(\ZZ/3)^4&(\ZZ/2)^2+(\ZZ/3)^2&&&& \\ \hline
(3,3,1)&&\ZZ/3&(\ZZ/2)^3+(\ZZ/3)^2&(\ZZ/2)^2+\ZZ/3&&&& \\ \hline
(4,2,1)&&&(\ZZ/2)^2+\ZZ/3&\ZZ/2&&&& \\ \hline
(3,3,2)&&&(\ZZ/2)^8+(\ZZ/3)^4&(\ZZ/2)^{10}+(\ZZ/3)^6&(\ZZ/2)^{11}+(\ZZ/3)^4&&& \\ \hline
(4,2,2)&\ZZ/2&&(\ZZ/2)^6+(\ZZ/3)^2&(\ZZ/2)^7+(\ZZ/3)^3&(\ZZ/2)^8+(\ZZ/3)^2&&& \\ \hline
(4,3,1)&&&(\ZZ/2)^4+\ZZ/3&(\ZZ/2)^6+\ZZ/3&(\ZZ/2)^7+\ZZ/3&&& \\ \hline
(5,2,1)&&&(\ZZ/2)^2&(\ZZ/2)^2&(\ZZ/2)^3&&&\\ \hline
(3,3,3)&&\ZZ/3&(\ZZ/2)^{12}+\ZZ/3&(\ZZ/2)^{25}+(\ZZ/3)^{10}&(\ZZ/2)^{64}+(\ZZ/3)^{18}&(\ZZ/2)^{28}+(\ZZ/3)^6&& \\ \hline
(4,3,2)&&&(\ZZ/2)^9+\ZZ/3&(\ZZ/2)^{19}+(\ZZ/3)^5&(\ZZ/2)^{50}+(\ZZ/3)^{10}&(\ZZ/2)^{22}+(\ZZ/3)^3&&\\ \hline
(4,4,1)&&&(\ZZ/2)^4&(\ZZ/2)^{10}+\ZZ/3&(\ZZ/2)^{29}+(\ZZ/3)^2&(\ZZ/2)^{14}+\ZZ/3&&\\ \hline
(5,2,2)&&&(\ZZ/2)^6&(\ZZ/2)^{10}+\ZZ/3&(\ZZ/2)^{28}+(\ZZ/3)^3&(\ZZ/2)^{12}&&\\ \hline
(5,3,1)&&&(\ZZ/2)^4&(\ZZ/2)^7&(\ZZ/2)^{21}+\ZZ/3&(\ZZ/2)^{10}&&\\ \hline
(6,2,1)&&&(\ZZ/2)^2&(\ZZ/2)^2&(\ZZ/2)^6&(\ZZ/2)^3&&\\ \hline
(5,3,2)&&&(\ZZ/2)^9&(\ZZ/2)^{21}+\ZZ/3&(\ZZ/2)^{90}+(\ZZ/3)^7&(\ZZ/2)^{94}+(\ZZ/3)^4&(\ZZ/2)^{62}+(\ZZ/3)^3&\\ \hline
(5,4,1)&&&(\ZZ/2)^4&(\ZZ/2)^{10}&(\ZZ/2)^{48}+\ZZ/3&(\ZZ/2)^{55}+\ZZ/3&(\ZZ/2)^{36}+\ZZ/3&\\ \hline
(6,2,2)&(\ZZ/2)^2&&(\ZZ/2)^8&(\ZZ/2)^{10}&(\ZZ/2)^{41}+\ZZ/3&(\ZZ/2)^{39}&(\ZZ/2)^{30}&\\ \hline
(6,3,1)&&&(\ZZ/2)^4&(\ZZ/2)^7&(\ZZ/2)^{27}&(\ZZ/2)^{29}&(\ZZ/2)^{21} &\\ \hline
(7,2,1)&&&(\ZZ/2)^2&(\ZZ/2)^2&(\ZZ/2)^6&(\ZZ/2)^6&(\ZZ/2)^6&\\ \hline
(7,3,1)&&&(\ZZ/2)^4&(\ZZ/2)^7&(\ZZ/2)^{27}&(\ZZ/2)^{36}&(\ZZ/2)^{63}& (\ZZ/2)^{27}\\\hline
\end{array}
$}
\end{table}


\subsection{General statements and questions} 

\subsubsection{Torsion in $\bar B_1(A_{n,k}(\Bbb Z))$}

\begin{thm}\label{b1barsup}
Let $k\ge 1$, $n+k\ge 2$, and $m_1,...,m_{n+k}>0$. Then 
the torsion in $\bar B_1(A_{n,k}(\Bbb Z))[\bold m]$
is $(\Bbb Z/2)^{2^{n+k-2}}$ if all $m_i$ are even, and zero 
otherwise. 
\end{thm}

\begin{rem}
We see that the behavior of torsion is different from the case $k=0$, when there is torsion of all orders. 
\end{rem}

\begin{proof}
Let $x_1,...,x_n,y_1,...,y_k$ be the even and odd generators, respectively, 
and let $\Omega(\Bbb Z^{n|k})$ be the algebra of differential forms on these variables
over $\Bbb Z$ (so the elements $x_i$ and $dy_j$ are even, while $dx_i$ and $y_j$ are odd).
Then similarly to Section 3, we have a linear isomorphism 
$\varphi: A_{n,k}(\Bbb Z)/M_3\to \Omega(\Bbb Z^{n|k})^{ev}$, which maps 
$L_2$ into exact forms of positive degree. However, the image
is not the entire set of exact positive forms; for instance, 
if $y$ is an odd variable, then $(dy)^{2r}=d(y(dy)^{2r-1})$ is an exact form but 
is not in the image of $L_2$. On the other hand, $2(dy)^{2r}$ is  
$\varphi([y,y^{2r-1}])=\varphi(2y^r)$. More generally, it is easy to show 
as in Section 3 that if $\omega$ is a positive exact form then $2\omega\in \varphi(L_2)$.  
This shows that we have an exact sequence 
$$
0\to K\to \bar B_1(A_{n,k}(\Bbb Z))\to \Omega^{ev}/\Omega^{ev}_{ex}\to 0,
$$
where $K$ is a vector space over $\Bbb F_2$. 

Now, as in the even case, we have 
$$
(\Omega^{ev}/\Omega^{ev}_{ex})[m_1,...,m_{n+k}]=
(\Omega^{ev}/\Omega^{ev,+}_{cl})[m_1,...,m_{n+k}]\oplus H^{ev,+}(\Bbb Z^{n|k})[m_1,...,m_{n+k}],
$$
where the first summand is a free group and the second summand 
is torsion. Furthermore, it is easy to show that for positive $m$, 
$H^i(\Bbb Z^{0|1})[m]=0$ (as the Poincar\'e lemma for odd variables holds over $\Bbb Z$), 
so using the K\"unneth formula as in the even case, 
we see that for $k>0$ we have $H^{ev,+}(\Bbb Z^{n|k})[m_1,...,m_{n+k}]=0$.
This means that the group $(\Omega^{ev}/\Omega^{ev}_{ex})[m_1,...,m_{n+k}]$ is free, and hence
the above exact sequence splits. 

It remains to compute the dimension of $K$. To do so, note that 
$\bar B_1(A_{n,k}(\Bbb Z))\otimes \Bbb F_2=\bar B_1(A_{n,k}(\Bbb F_2))$.
On the other hand, $\bar B_1(A_{n,k}(\Bbb F_2))=\bar B_1(A_{n+k}(\Bbb F_2))$ 
(the super-structure does not matter in characteristic $2$), 
so the dimension of the right hand side in multidegree $(m_1,...,m_{n+k})$ 
is known from Section 3. Finally, the rank of the free part $\bar B_1(A_{n,k}(\Bbb Z))[\bold m]$ is known from 
\cite{BJ}. Subtracting the two, we get the statement. 
\end{proof}

\subsubsection{Torsion in $B_1(A_{n,k}(\Bbb Z))$}

In the case $k>0$, unlike the purely even case $k=0$, the group 
$B_1(A_{n,k}(\Bbb Z))$ has torsion. Namely, recall that 
$B_1(A_n(\Bbb Z))$ has a basis consisting of cyclic words, i.e., words in the generators 
up to cyclic permutation. This remains true for $A_{n,k}$, except that 
some words turn out to be equal to minus themselves, thus being 2-torsion. 
It is clear that such words are exactly even powers of odd words. 
This implies the following result. 

\begin{prop}\label{torb1}
The group ${\rm tor}B_1(A_{n,k}(\Bbb Z))$ is a vector space over $\Bbb F_2$, whose multivariable Hilbert series 
is given by the formula 
$$
h_{n,k}(t_1,...,t_n,u_1,...,u_k)=
$$
$$
\sum_{\bold m: m_{n+1}+...+m_{n+k}\text{ is odd}} a_{m_1,...,m_{n+k}}
\frac{t_1^{2m_1}...t_n^{2m_n}u_1^{2m_{n+1}}...u_k^{2m_{n+k}}}
{1-t_1^{2m_1}...t_n^{2m_n}u_1^{2m_{n+1}}...u_k^{2m_{n+k}}},
$$
where $f(z):=\sum a_{m_1,...,m_{n+k}} z_1^{m_1}...z_{n+k}^{m_{n+k}}$
is the Hilbert series of the free Lie algebra on $n+k$ generators:
$$
f(z)=\sum_{d\ge 1}\frac{\mu(d)}{d}\log(1-z_1^d-...-z_{n+k}^d),
$$
where $\mu(d)$ is the M\"obius function. 

In particular, torsion in multidegree $\bold m$ is nonzero if and only if  
all $m_i$ are even, and \linebreak $(m_{n+1}+...+m_{n+k})/{\rm gcd}(\bold m)$ is odd. 
\end{prop}

\begin{ex}\label{twovar} Let $\nu(m)$ be the maximal power of $2$ dividing $m$. 
For $n=0$, $k=2$, the torsion in $B_1$ in bidegree $(m_1,m_2)$ is nonzero 
iff $m_1$ and $m_2$ are even, and $\nu(m_1)\ne \nu(m_2)$. For $n=1$ and $k=1$, 
the torsion in $B_1$ in bidegree $(m_1,m_2)$ is nonzero 
iff $m_1$ and $m_2$ are even, and $\nu(m_1)\ge \nu(m_2)$. 
\end{ex}

\begin{proof}
It is clear that any nontrivial word is a power of a non-power word (i.e. a word which is not a power of another word) 
in a unique way, and words generating 2-torsion are exactly the even powers of non-power words of odd
total degree with respect to the odd generators (if $z$ is such a non-power word, then 
$2z^{2l}=[z,z^{2l-1}]$).  

Also, it is well known that the generating function for non-power words modulo cyclic permutation
is the same as the Hilbert series of the free Lie algebra. 
Indeed, if the generating function for non-power words modulo cyclic permutation
is $f(v)=\sum a_{\bold m}v^{\bold m}$, where $v^{\bold m}=v_1^{m_1}...v_{n+k}^{m_{n+k}}$, 
then 
$$
\sum_{\bold m} a_{\bold m}(\sum m_j)\frac{v^{\bold m}}{1-v^{\bold m}}=\frac{\sum v_j}{1-\sum v_j},
$$
since any nontrivial word is uniquely a power of a non-power word. 
(The factor $\sum m_j$ appears because non-power words are considered up to cyclic permutations, 
which act freely on them).
Applying the inverse of the Euler operator, $E^{-1}$, which divides terms of total degree $N$ by $N$, we get
$$
\sum_{\bold m} a_{\bold m}\log (1-v^{\bold m})=\log(1-\sum v_j),
$$
or 
$$
\prod_{\bold m} (1-v^{\bold m})^{a_{\bold m}}=1-\sum v_j,
$$
which implies that $a_{\bold m}$ are the dimensions of the components of the free Lie algebra. 

These two statements imply the proposition.
\end{proof}

\subsubsection{Torsion in $B_2(A_{n,k}(\Bbb Z))$}

\begin{prop}\label{no4tor}
If $k\ge 1$ and $m_i>0$ then torsion in $B_2(A_{n,k}(\Bbb Z))[m_1,...,m_{n+k}]$ 
is a vector space over $\Bbb F_2$, which is a quotient of 
$\Omega^{odd}_{ex}(\Bbb Z^{n|k})[\bold m]\otimes {\Bbb F}_2$. In particular, there is no 4-torsion,
and $B_2(A_{n,k}(\Bbb Z[1/2]))[m_1,...,m_{n+k}]$
is a free abelian group.  
\end{prop}

\begin{proof}
It is shown in the proof of Theorem \ref{b1barsup} that 
$\bar B_1(A_{n,k}(\Bbb Z))[\bold m]$ is a quotient of 
$\Omega^{ev}(\Bbb Z^{n|k})[\bold m]/2\Omega^{ev}_{ex}(\Bbb Z^{n|k})[\bold m]$.
Therefore, analogously to the proof of Theorem \ref{upbo} one shows that  
$B_2(A_{n,k}(\Bbb Z))[\bold m]$ is a quotient of 
$\Omega^{odd}(\Bbb Z^{n|k})[\bold m]/2\Omega^{odd}_{ex}(\Bbb Z^{n|k})[\bold m]$.
But as explained in the proof of Theorem \ref{b1barsup}, 
the De Rham cohomology groups vanish in positive multidegrees. 
Hence, the subgroup $\Omega^{odd}_{ex}(\Bbb Z^{n|k})[\bold m]$
is saturated. So the torsion in $B_2(A_{n,k}(\Bbb Z))[\bold m]$ is a quotient of 
$\Omega^{odd}_{ex}(\Bbb Z^{n|k})[\bold m]/2$, as desired. 
\end{proof}

Let us now discuss the 2-torsion in $B_2(A_{n,k}(\Bbb Z))$ (which by Proposition \ref{no4tor} 
is all the torsion in positive degrees for $k>0$). 
Arguing as in Lemma \ref{lem:Universal}, we see that for any abelian groups 
$\Bbb Z^N\supset A\supset B$, 
we have an exact sequence 
$$
{\rm tor}_p(\Bbb Z^N/A)\to (A/B)\otimes \Bbb F_p\to A_p/B_p\to 0,
$$
which implies that 
$$
\dim(A_p/B_p)\le \dim ((A/B)\otimes \Bbb F_p)\le \dim(A_p/B_p)+\dim {\rm tor}_p(\Bbb Z^N/A).
$$
Setting $p=2$, $\Bbb Z^N=L_1[\bold m]$, $A=L_2[\bold m]$, and $B=L_3[\bold m]$, 
this becomes
$$
\dim B_2(A_{n+k}(\Bbb F_2))[\bold m]\le \dim B_2(A_{n,k}(\Bbb Z))[\bold m]\otimes \Bbb F_2\le \dim B_2(A_{n+k}(\Bbb F_2))[\bold m]+\dim {\rm tor}_2(B_1(A_{n,k}(\Bbb Z))[\bold m].
$$
Subtracting $\dim B_2(A_{n+k}(\Bbb Q))$, we get 

\begin{prop}\label{ineq}
$$
\dim {\rm tor}_2B_2(A_{n+k}(\Bbb Z))[\bold m]\le \dim {\rm tor}_2(B_2(A_{n,k}(\Bbb Z)))[\bold m]\le 
$$
$$
\le\dim {\rm tor}_2B_2(A_{n+k}(\Bbb Z))[\bold m]+
\dim {\rm tor}_2(B_1(A_{n,k}(\Bbb Z))[\bold m].
$$
\end{prop}

\begin{rem}
1. Both inequalities in Proposition \ref{ineq} can be strict. 
For instance, for $B_2(A_{2,1}(\Bbb Z))[2,2,2]$, they look like $1\le 2\le 3$. 

2. The upper bound of Proposition \ref{ineq} is asymptotically bad -- it grows exponentially with degrees, 
while the bound of Proposition \ref{no4tor} is uniform in degrees. However, in particular degrees 
the bound of Proposition \ref{ineq} may be better. For example, for $n=k=1$ and $\bold m=(2,4)$, 
then there is a unique odd exact form up to scaling, namely $xdx(dy)^4=d(xydx(dy)^3)$, so 
Proposition \ref{no4tor} gives the bound of $1$, while Proposition \ref{ineq} implies that 
${\rm tor}B_2(A_{1,1}(\Bbb Z))[2,4]=0$. 
\end{rem} 

Proposition \ref{ineq} together with Proposition \ref{torb1} 
implies the following corollary. 

\begin{cor}
If the group $B_2(A_{n,k}(\Bbb Z))[\bold m]$ has nontrivial 2-torsion, 
then all $m_i$ are even. Moreover, unless 
$(m_{n+1}+...+m_{n+k})/{\rm gcd}(\bold m)$ is odd, this 2-torsion 
depends only on $n+k$ (i.e., is the same as the 2-torsion in $B_2(A_{n+k}(\Bbb Z))$).   
\end{cor}

\begin{cor}
If $n+k=2$ then necessary conditions 
for a nontrivial 2-torsion are the conditions of Example 
\ref{twovar}.  
\end{cor}

Note that this agrees with the results of computer calculations.

\subsubsection{Results and questions on torsion in higher $B_\ell$} 

\begin{prop}\label{nontor}
Nontrivial torsion in $B_\ell(A_{n,k}(\Bbb Z[1/2]))[m_1,...,m_{n+k}]$ implies \linebreak $\ell \le m_1 +  \ldots + m_{n+k} - 2$.
\end{prop}

\begin{proof}
Same as in Remark \ref{weakerst} (but using the free Lie superalgebra).
\end{proof}

\begin{rem}
1. Proposition \ref{nontor} seems to be true over $\Bbb Z$ but it seems that the proof would need to be modified, since 
Lie superalgebras don't behave well in characteristic $2$. 

2. We see from the tables that there {\it can} be torsion 
if $\ell=m_1 +  \ldots + m_{n+k}-2$ (this does not seem to occur 
in the even case).  
\end{rem}

\begin{prop}
If $n+k=2$, there is no torsion in $B_\ell[m_1,m_2]$ if either $m_1$ or $m_2$ equals $1$.
\end{prop}

\begin{proof} Denote the generators of $A_{n,k}(\Bbb Z)$ by $x,y$, and consider 
$A_{n,k}(\Bbb Z)[m,1]$. A basis of this group is the collection of elements
$x^iyx^j$, where $i+j=m$. 

Consider first the case when $x$ is even. 
Let us identify $A_{n,k}(\Bbb Z)[m,1]$ with the space of polynomials of a variable $z$ 
with integer coefficients and degree $\le m$, by $x^iyx^j\mapsto z^i$. 
In this case, the operation of bracketing with $x^l$ acting from $A_{n,k}(\Bbb Z)[m-l,1]$ to 
$A_{n,k}(\Bbb Z)[m,1]$ corresponds to the operator $P\mapsto (z^l-1)P$ on polynomials. 
This implies that $L_{i+1}(A_{n,k}(\Bbb Z))[m,1]$ is identified with 
polynomials of degree $\le m$ which are divisible by $(z-1)^i$. So we see that 
$B_{i+1}(A_{n,k}(\Bbb Z))[m,1]=\Bbb Z$ for $i=0,1.,,,m$, and zero for larger $i$, 
so there is no torsion. 

If $x$ is odd, the argument is similar. Namely, if $y$ is even, 
we should use the assignment $x^iyx^j\mapsto (-1)^{ij}z^i$, 
while if $y$ is odd, we should use the assignment 
$x^iyx^j\mapsto (-1)^{i(j+1)}z^i$. 
\end{proof} 

\begin{thm}
The dimensions of $B_i(A_{n,k}(\mathbb{F}_p))[\mathbf{m}]$ for $i\ge 2$ are bounded from above by a constant $K_i$ 
independent of $ \mathbf{m}$ and $p$ (but depending on $n,k$).
\end{thm}

\begin{proof}
Since $[ab,c]+[bc,a]+[ca,b]=0$, we have 
$$
[ab,[c,d]]+[bc,[a,d]]+[ca,[b,d]]=[c,[ab,d]]+[a,[bc,d]]+[b,[ca,d]].
$$
Let $A=A_{n,k}(\Bbb F_p)$. Suppose we know that $L_i=[A_{\le r},L_{i-1}]$, where 
$A_{\le r}$ is the part of $A$ of degree $\le r$. Let $u$ be a word of degree $\ge 2r+2$. 
Then we can write $u$ as $ab$, where the degrees of $a$ and $b$ are $\ge r+1$. 
Since $L_i$ is the sum of the spaces $[c,L_{i-1}]$, where $c$ has degree $\le r$, 
we find that $[u,L_i]$ is contained in the sum of spaces $[v,L_i]$, where $v$ has degree 
$\le 2r+1$. This implies that $L_{i+1}=[A_{\le 2r+1},L_i]$. 
Since $L_2=[A_{\le 1},L_1]$ by Lemma \ref{degg1}, we obtain by induction in $i$ that 
$L_{i+1}=[A_{\le 2^i-1},L_i]$ and hence $B_{i+1}=[A_{\le 2^i-1},B_i]$. This implies that if $K_i$ is defined, then 
one can take $K_{i+1}=K_i\dim A_{\le 2^i-1}$. 

It remains show that $K_2$ is defined. To show this, recall that $B_2=\sum [x_i,\bar B_1]$. 
On the other hand, by the results of \cite{BJ} and Theorem \ref{b1barsup}, 
the dimensions of the homogeneous components of $\bar B_1$ are bounded by some constant 
$K$. So we can take $K_2=(n+k)K$, and we are done. 
\end{proof} 
 
\begin{rem}
The bound $K_i$ given by this proof is very poor 
(it grows like $2^{i^2/2}$). For 
$p\ge 3$, one can get a much better 
bound, which grows exponentially with $i$.  
Namely, since $L_\ell(A_{n,k}(\Bbb Z[1/2]))[\ell]$ is a saturated subgroup of 
$A_{n,k}(\Bbb Z[1/2])[\ell]$, it follows from Lemma 3.1 in \cite{AJ}
that over any ring $R$ containing $1/2$, one has 
$$
[(ab+ba)c,B_i]\subset [Q,B_i],
$$
where $Q$ is the span of at most quadratic monomials in $a,b,c$
(where $a,b,c\in A_{n,k}$). This implies that if $E$ is a complement to the 
image in $\bar B_1$ of $1$ and all elements $(ab+ba)c$ where $a,b,c$ are of positive degree, then 
$B_{i+1}=[E,B_i]$. But for $E$ we can take the span of all words of degrees $1,2$
in the generators. The number of such words is $\frac{(n+k)^2}{2}+\frac{3n+k}{2}$. 
Thus, for $p\ge 3$ we can take 
$$
K_i=(\frac{(n+k)^2}{2}+\frac{3n+k}{2})^{i-2}(n+k)K.
$$
for $i\ge 2$. 

We expect that by a suitable strengthening of Lemma 3.1 in \cite{AJ}, 
this argument can be adapted to get an exponentially growing bound 
$K_i$ also in the case $p=2$.     
 \end{rem}

Let us now state some questions that are motivated by the data in the tables. 

{\bf Question 3.} Suppose $k\ge 1,n+k\ge 2$. Can $B_i(A_{n,k}(\Bbb Z))$ 
have $p$-torsion with $p>n+k$? For example, if $n+k=2$ (i.e., $(n,k)=(0,2)$ or $(1,1)$), can 
$B_i$ have $p$-torsion for odd $p$? Is the torsion in $B_i$ in this case a vector space over $\Bbb F_2$? 

{\bf Question 4.} Is there torsion in $B_3(A_{0,2}(\Bbb Z))$? 

{\bf Question 5.} If $n+k=3$, can there be torsion in $B_3(A_{n,k}(\Bbb Z))[m_1,m_2,m_3]$
($m_1,m_2,m_3> 0$) other than 3-torsion? 

\begin{rem} We note that, as follows from the above tables, 
it is not true in the supercase that if a torsion element $T$ in $B_i$ has degree $m$ 
with respect to some generator, then the order of $T$ divides $m$.
For example, for $n=1,k=2$, we have a 3-torsion element in multidegree $(1,2,2)$.  
\end{rem}

\subsection{Lower central series and modular representations} 

Let us discuss the connection of the theory of lower central series with modular 
representation theory. It is clear that the algebraic supergroup scheme $GL(n|k)(\Bbb Z)$ acts 
on each $B_i(A_{n,k}(\Bbb Z))[m]$ (elements of total degree $m$), 
with decomposition into multidegrees being the weight decomposition. 
Moreover, for each prime $p$, the $p$-torsion in this representation is a subrepresentation, which actually 
factors through $GL(n|k)(\Bbb F_p)$. Moreover, the tensor product of this representation with the algebraic closure 
$\overline{\Bbb F_p}$ naturally extends to a representation of $GL(n|k)(\overline{\Bbb F_p})$. 

In the purely even case, from computer data it appears that this modular representation 
is always trivial at the Lie superalgebra level. However, in the supercase we see nontrivial 
Lie superalgebra representations. 

For instance, for $n=1$ and $k=2$, we have a 3-torsion element in $B_3$ in multidegree $(1,2,2)$,
which is unique up to scaling in degree $5$. Its degrees are not divisible by $3$, so 
it defines a 1-dimensional representation of $GL(1|2)$ in characteristic $3$
which is nontrivial at the Lie superalgebra level.  
It is easy to see that at the Lie superalgebra level, 
this is in fact the well-known Berezinian representation 
with weight $(1,-1,-1)$ (which is the same as $(1,2,2)$ in characteristic $3$). 

Now consider the 3-torsion in $B_3$ for $n=0$ and $k=3$, which gives representations of $GL(3)$ 
(see the last table in Subsection 6.1). In this case, 
in degree $5$ we have three independent torsion elements, 
in multidegrees $(1,2,2)$, $(2,1,2)$ and $(2,2,1)$, respectively. 
Clearly, they generate a copy of the representation $V^*\otimes (\wedge^3V)^{\otimes 2}$, where
$V$ is the 3-dimensional vector representation spanned by the generators
of $A_{0,3}$. 

In degree $6$, there is an $8$-dimensional representation, 
with weights being all the permutations of $(1,2,3)$, as well as $(2,2,2)$ (with multiplicity $2$). 
This representation has the same character as ${\mathfrak{sl}}(V)\otimes (\wedge^3V)^{\otimes 2}$.
Note that we cannot claim without additional argument that they are isomorphic,
since ${\mathfrak{sl}}(V)$ is reducible ($1\in {\mathfrak{sl}}(V)$ in characteristic $3$ as $V$ is $3$-dimensional).  
All we can say is that its composition series consists of 
a $7$-dimensional irreducible representation and the 
1-dimensional representation $(\wedge^3V)^{\otimes 2}$. 

In degree $7$, we have a $6$-dimensional representation (whose weights are 
permutations of $(1,3,3)$ and $(2,2,3)$, with 1-dimensional weight spaces), 
which is clearly $S^2V^*\otimes (\wedge^2V)^{\otimes 3}$. 
Finally, we have a representation $(\wedge^3V)^{\otimes 3}$ (trivial at the Lie algebra level) 
sitting in degree $9$.  

Actually, by analogy with \cite{FS}, for $p>2$ 
the $p$-torsion in $B_i$ is a representation of the 
Lie superalgebra $W_{n|k}$ of supervector fields in characteristic $p$
(this action is compatible with the action of $GL(n|k)$). 
In particular, we have an action of the Lie superalgebra 
${\mathfrak{pgl}}(n+1|k)$ which sits inside $W(n|k)$ as the Lie algebra of 
the supergroup of projective transformations. Looking at the action of this subalgebra, 
we get more information. 

For instance, consider the case $n=0$, $k=3$. 
In this case, the 3-torsion in $B_3$
tensored with the algebraic closure $\overline{\Bbb F}_p$
carries an action of the Harish-Chandra pair $({\mathfrak{sl}}(1|3),GL(3))$ 
(which is equivalent to an action of the supergroup $SL(1|3)$).
Note that the trivial 1-dimensional representation of ${\mathfrak{sl}}(1|3)$ has 
a family of lifts $\chi^{\otimes r}$ to $SL(1|3)$, parametrized by integers $r$;
namely $\chi|_{GL(3)}=(\wedge^3V)^{\otimes 3}$.  

Now consider the 3-torsion $X$ in degrees 5,6,7 in $B_3(A_{0,3}(\Bbb Z))$
tensored with the algebraic closure (so $\dim(X)=17$). 
Let us describe the composition series of $X$ as a representation of $SL(1|3)$. 
It is easy to see that the irreducible module for $SL(1|3)$ with lowest degree component 
$V^*\otimes (\wedge^3V)^{\otimes 2}$ is $\tilde V^*\otimes \chi$,
where $\tilde V$ is the 4-dimensional defining representation of $SL(1,3)$. 
As a $GL(3)$-module, this representation is $V^*\otimes (\wedge^2V)^{\otimes 2} \oplus (\wedge^3V)^{\otimes 2}$. 
Thus, we see that $\tilde V^*\otimes \chi$ 
is contained in the composition series of $X$.
The remaining $GL(3)$-modules are the 6-dimensional and 7-dimensional 
irreducible representations of $GL(3)$, and they have to comprise an irreducible $SL(1|3)$-module, 
since any $SL(1|3)$ module with the trivial action of the odd generators of the Lie algebra 
must be trivial for the whole Lie algebra. Thus, we see that the composition series of $X$ has length 2 and consist of a 4-dimensional and a 
13-dimensional irreducible representations, living n degrees 5,6 and 6,7, respectively. 
They have Hilbert series $3t^5+t^6$ and $7t^6+6t^7$. 

In a similar way one should be able to study the 3-torsion  
in $B_i(A_{n,k})$ for $i>3$ (as well as higher $p$-torsion). 

\section{Further work}

Besides attempting to answer the questions stated above and demystify the enigmatic pattern of torsion in
$B_i(A_n)$ and more generally $B_i(A_{n,k})$, we can point out the following interesting directions of future research.

1. Over $\mathbb{Z}[1/2]$, the Lie algebra $W_n$ of vector fields on the $n$-dimensional space
acts on $\bar B_1(A_n)$ and $B_i(A_n)$ for $i\ge 2$, similarly to \cite{FS}
(in characteristic 2, this can be generalized but a  
modification is needed).
It would be interesting to study this action and to apply representation theory of $W_n$ to
the study of the torsion in $B_i(A_n)$. In particular, we conjecture that
the Lie algebra $W_n$ acts trivially
on the torsion in $B_i(A_n)$
(by analogy with the fact that vector fields
on a manifold act trivially on its cohomology).
Note that this conjecture implies the conjecture that the order
of a torsion element divides its degree with respect
to each variable $x_i$, by considering the vector field $x_i\frac{\partial}{\partial x_i}$.

We expect that in this study over $\mathbb{F}_p$, the restricted structure
of the Lie algebra $W_n$ will play a role.

Also it would be interesting to extend this approach to the supercase, where 
there is an action of the Lie superalgebra $W_{n|k}$. 

2. It would be interesting to generalize the results
of this paper from the quotients $B_i$ to the quotients
$N_i:=M_i/M_{i+1}$. (We note that a search in low
degrees found no torsion in these quotients).

3.  It would be interesting to extend the results of this paper to algebras with relations, along the lines
of the appendix to \cite{DKM} and of \cite{BB}.


\pagestyle{empty}\singlespace

\end{document}